\documentclass{amsart}

\usepackage{amsmath, amsthm, amssymb, mathrsfs, mathtools}
\usepackage{etoolbox}
\apptocmd{\sloppy}{\hbadness 10000\relax}{}{}
\apptocmd{\sloppy}{\vbadness 10000\relax}{}{}
\usepackage{enumerate, enumitem, hyperref, xcolor}
\usepackage{esint}
\usepackage{comment}
\usepackage{braket}
\numberwithin{equation}{section}

\theoremstyle{plain}
\newtheorem{theorem}{Theorem}[section]
\newtheorem{lemma}[theorem]{Lemma}
\newtheorem{corollary}[theorem]{Corollary}
\newtheorem{proposition}[theorem]{Proposition}

\theoremstyle{definition}
\newtheorem{definition}[theorem]{Definition}
\newtheorem{remark}[theorem]{Remark}
\newtheorem{example}[theorem]{Example}

\numberwithin{equation}{section}

\newcommand{\scR}{\mathscr{R}}
\newcommand{\scL}{\mathscr{L}}

\newcommand{\scB}{\mathscr{B}}
\newcommand{\scQ}{\mathscr{Q}}

\newcommand{\D}{\mathrm{D}}

\newcommand{\cW}{\mathcal{W}}
\newcommand{\cH}{\Haus}

\newcommand{\R}{\mathbb{R}}

\newcommand{\N}{\mathbb{N}}
\newcommand{\X}{\mathbb{X}}

\newcommand{\bz}{\mathbf{0}}
\newcommand{\Haus}{\mathcal{H}}
\DeclareMathOperator{\appTanAW}{AppTan_{AW}}
\DeclareMathOperator{\appTanGH}{AppTan_{GH}}
\DeclareMathOperator{\Tan}{\mathrm{Tan}}
\DeclareMathOperator{\TanAW}{Tan_{AW}}
\DeclareMathOperator{\TanGH}{Tan_{GH}}

\newcommand{\AW}{\mathrm{AW}}

\DeclareMathOperator{\diam}{\mathop\mathrm{diam}}
\DeclareMathOperator{\dist}{\mathop\mathrm{dist}}
\newcommand{\GH}{{\mathrm{GH}}}
\newcommand{\pGH}{{\mathrm{pGH}}}
\newcommand{\ex}{\mathop\mathrm{ex}}

\newcommand{\sD}{\mathrm{D}}

\newcommand{\pcap}{\mathrm{cap}}

\newcommand{\td}[1]{\tilde{#1}}

\title{Tangents to Lipschitz and Sobolev images}
\author{Matthew Badger}
\author{Jared Krandel}
\author{Vyron Vellis}
\thanks{M.B.~was partially supported by NSF grant 2154047. J.K.~was partially supported by NSF grants 1763973 and 1928930 (the latter while in residence at the Simons Laufer Mathematical Sciences Institute during Fall 2024). V.V.~was partially supported by NSF grant 2154918.}
\date{January 29, 2026}
\address{Department of Mathematics\\ University of Connecticut\\ Storrs, CT 06269-1009}
\email{matthew.badger@uconn.edu}
\address{Department of Mathematics, University of California, Davis, Davis, CA 95616}
\email{jkrandel@ucdavis.edu}
\address{Department of Mathematics\\ The University of Tennessee\\ Knoxville, TN 37966}
\email{vvellis@utk.edu}

\begin{document}
	
	\begin{abstract} We develop geometric versions of Rademacher and Calderon type differentiability theorems in two categories. A special case of our results is that for any Lipschitz or continuous $W^{1,p}$ Sobolev map $f$ from $[0,1]^n$ into a Euclidean space with $p>n$, the image $f([0,1]^n)$ has a unique tangent set (Attouch-Wets convergence) at almost every point with respect to the $n$-dimensional Hausdorff measure. In the analogous case when $f$ is a continuous $N^{1,p}$ map from $[0,1]^n$ into a metric space, we show that the image $f([0,1]^n)$ has a unique metric tangent (Gromov-Hausdorff convergence) almost everywhere. These results complement, but are distinct from Federer's theorem on existence and uniqueness of approximate tangents of $n$-rectifiable sets in $\R^d$. We show that approximate tangents to Sobolev images can be upgraded to Attouch-Wets or Gromov-Hausdorff tangents by first proving that the \emph{$n$-packing content} of Sobolev images is finite, then proving that the inability to upgrade on a set of positive measure implies infinite packing content.
	\end{abstract}
	
	\maketitle
	
	\section{Introduction}
	Rademacher's theorem, an important result in geometric measure theory, states that if $\Omega\subset\R^n$ is open and $f: \Omega\to\R^d$ is Lipschitz, then $f$ is differentiable at almost every point of $\Omega$. Applications and generalizations of Rademacher's theorem are in several facets of analysis. First, Stepanov's theorem states that if a map $f:\Omega\to\R^d$ on an open set $\Omega\subset\R^n$ is differentiable at almost every point $a\in \Omega$ at which $\limsup_{x\to a} |f(x)-f(a)|/|x-a|<\infty$; see \cite[Theorem 3.1.9]{Federer}.  Second, Aleksandrov's theorem states that a convex function $f:\R^n \to \R$ is almost everywhere twice differentiable \cite[Theorem 3.11.2]{convex-book}. Third, Pansu's theorem \cite{Pansu} establishes a Rademacher-type theorem where the domain of the Lipschitz map is an open set of a Carnot group. Fourth, Kirchheim \cite{Ki94} proved a Rademacher-type theorem where the co-domain of $f$ is a metric space $X$; see Theorem \ref{t:kirch}. Finally, applications of Rademacher's theorem can be found in analysis on metric spaces (bi-Lipschitz non-embeddability of the Heisenberg group \cite{Semmes}), analysis of PDEs (solutions to certain nonlinear wave equation in one spatial dimension decay to zero \cite{TaoLindblad}), and geometric measure theory (the Besicovitch projection theorem in the plane and the general theory of rectifiability \cite{Bes39}). A far-reaching extension of Rademacher's theorem to metric measure spaces has been obtained by Cheeger \cite{Cheeger}.
	
	Calderon \cite{Calderon}, using the Sobolev embedding theorem, extended Rademacher's theorem by proving that if $\Omega \subset \R^n$ is an open set, then every continuous function in the Sobolev class $W^{1,p}(\Omega,\R^d)$ is almost everywhere differentiable when $p\in (n,\infty]$. Recall that a function $f:\Omega\to\R^d$ is in $W^{1,p}(\Omega,\R^d)$ for some $p\geq 1$ if $f\in L^p(\Omega)$ and $f$ has distributional partial derivatives in $L^p(\Omega)$. Recall that a continuous function $f: \Omega \to \R^d$ is Lipschitz if and only if $f\in W^{1,\infty}(\Omega,\R^d)$ \cite[Theorem 6.12]{Heinonen01}.
	
	The goal of this paper is to establish geometric versions of Rademacher's theorem and Calderon's theorem, replacing functions by their images and derivatives by tangent spaces that describe the infinitesimal structure of the image of the map. Let us first define a notion of tangent. Following \cite{BL15}, given a closed set $X\subset \R^n$ and a point $x\in X$, we say that a closed set $T$ is a \emph{tangent of $X$ at $x$} if there exists a sequence of positive scales $r_j$ that go to zero such that the blow-up sets $r_j^{-1}(X-x)$ converge to the set $T$ in the Attouch-Wets topology. We denote by $\TanAW(X,x)$ the collection of all tangents of $X$ at $x$. The space $\TanAW(X,x)$ is nonempty and if $T$ is in $\TanAW(X,x)$, then $\lambda T$ is in $\TanAW(X,x)$ for every $\lambda>0$. See \textsection \ref{sec:euc_tang} for the precise definition and further discussion.
	
	The third named author and Shaw showed that if $f:[0,1] \to \R^n$ is Lipschitz, then for $\Haus^1$-almost every point $x \in f([0,1])$, the tangent space $\TanAW(f([0,1]),x)$ contains exactly one element and this element is a straight line \cite[Theorem 1.1]{ShawV}. The first theorem of this paper extends this result to all dimensions as well as to Sobolev maps.
	
	\begin{theorem}\label{thm:main-euc}
		Let $n\geq 1$ and $p\in (n,\infty]$. For every continuous map $f \in W^{1,p}(\mathbb{B}^n,\R^d)$, the image $f(\mathbb{B}^n)$ has finite $\Haus^n$ measure, and at $\Haus^n$-a.e. $x \in f(\mathbb{B}^n)$ there exists an $n$-plane $V_x\subset\R^d$ such that $\TanAW(f(\mathbb{B}^n),x) =\{V_x\}$.
	\end{theorem}
	
	In our second result, we develop a version of Theorem \ref{thm:main-euc} for metric targets. To state it, we first need to define tangents and Sobolev maps in metric spaces.
	
	There have been several approaches to generalizing classical Sobolev spaces to metric measure spaces, including but not limited to Haj{\l}asz \cite{Haj96}, Cheeger \cite{Cheeger}, and Shanmugalingam \cite{Sha00}. Here we are concerned with the latter notion. Given $p \in [1,\infty]$, a domain $\Omega\subset\R^n$, and a metric space $X$, we say that $f :\Omega \to X$ is in the \emph{Newtonian-Sobolev class} $N^{1,p}(\Omega,X)$ if $f\in L^p(\Omega,X)$ and if there exists a Borel function $\rho \in L^p(\Omega)$ such that $\rho\geq 0$ and for $p$-almost every rectifiable curve $\gamma : [0,1] \to \Omega$,
	\[ d_X(f(\gamma(1)), f(\gamma(0))) \leq \int_{\gamma}\rho\, ds.\]
	See Definition \ref{def:metricSobolev} for precise statement.
	
	Tangents can be defined in any complete pointed metric space $(X,d_X,x)$; the primary difference is that convergence of pointed metric spaces $(X,r_j^{-1}d_X,x)$ is done in the Gromov-Hausdorff sense (see Definition \ref{def:GHdist}). The associated tangent space is denoted by $\TanGH(X,x)$. Assuming that $E\subset \R^n$ and $x\in E$, a key difference between spaces $\TanGH(E,x)$ and $\TanAW(E,x)$ is that in the latter, two isometrically equivalent elements are indistinct while in the former they are not. For example, if
	\[ \mathcal{S} = \{\bz\}\cup\{te^{i\log{t}} : t> 0 \} \subset\R^2\]
	is the logarithmic spiral, then $\TanAW(\mathcal{S},\bz) = \{e^{i\theta}\mathcal{S}:\theta\in [0,2\pi)\}$ while $\TanGH(\mathcal{S},\bz)$ contains only one element (which is the isometry class of $\mathcal{S}$). See also \cite[Lemma 3.1]{BET} for additional examples.
	On the other hand, $\TanGH(E,x)$ is the set of isometry classes of elements of $\TanAW(E,x)$ in the sense that every element in $\TanAW(E,x)$ is isometric to exactly one element in $\TanGH(E,x)$ and every element in $\TanGH(E,x)$ is isometric to some element in $\TanAW(E,x)$ (see Theorem \ref{t:two-tangents} for the precise statement).
	
\begin{figure}
 \centering
  \includegraphics[width=0.3\textwidth]{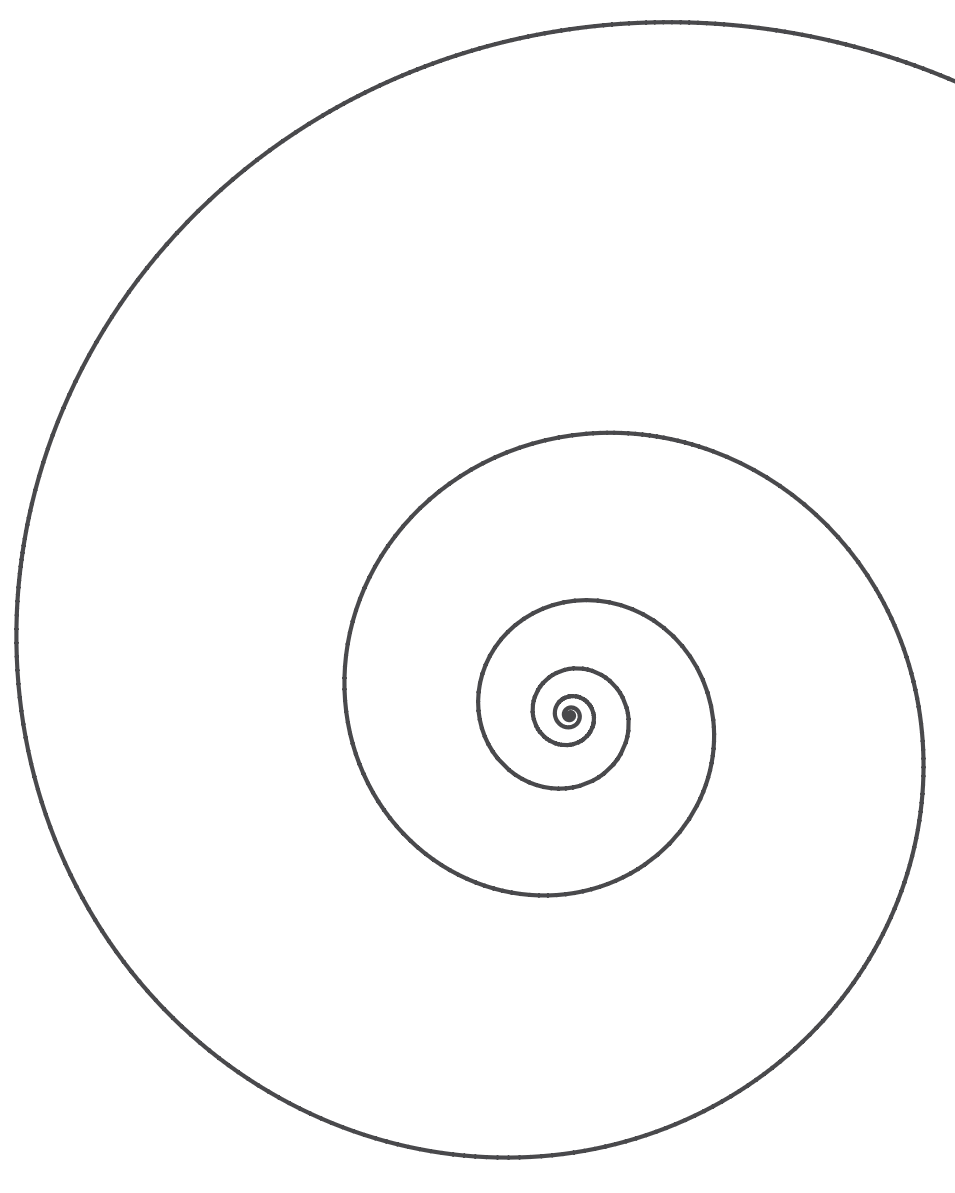}
 \caption{The logarithmic spiral $\mathcal{S}$.}
 \end{figure}

	The next result is the metric version of Theorem \ref{thm:main-euc}.
	
	\begin{theorem}\label{thm:main-met} Let $n\geq 1$ and let $p\in(n,\infty]$. For every metric space $X$ and continuous map $f \in N^{1,p}(\mathbb{B}^n,X)$, the image $f(\mathbb{B}^n)$ has finite $\Haus^n$ measure, and at $\Haus^n$ almost every $x \in f(\mathbb{B}^n)$ the tangent space $\TanGH(f(\mathbb{B}^n),x)$ is the pointed isometry class of an $n$-dimensional pointed Banach space $(\R^n,\|\cdot\|_x,\bz)$.
	\end{theorem}
	
	In both Theorem \ref{thm:main-euc} and Theorem \ref{thm:main-met}, we can replace $\cH^n$ with the packing $n$-measure $\mathcal{P}^n$. Additionally, both Theorem \ref{thm:main-euc} and Theorem \ref{thm:main-met} are special cases of the more general results Theorem \ref{t:euc-flat-tangents} and Theorem \ref{t:flat-tangents}. A collection of closely related results state that $n$-rectifiable Euclidean sets with finite Hausdorff $n$-measure have unique \emph{approximate tangent $n$-planes} almost everywhere; see Lemma \ref{l:apptan-planes-exist} for a proof of one version and \cite[Theorem 11.6, Remark 11.7]{Simon} and \cite[Theorem 15.19]{Ma95} for others. Roughly speaking, an $n$-plane $P$ is an approximate tangent of $E$ at $x$ if there exists $\tilde{E}\subset E$ containing $x$ such that the Hausdorff $n$-density of $E\setminus \tilde{E}$ at $x$ is zero and $P$ is an Attouch-Wets tangent of $\tilde{E}$ at $x$.
	A metric analogue of this result was given by Kirchheim \cite{Ki94}. Note that the criteria for the existence of an approximate tangent $n$-plane are weaker than for a ``true'' tangent plane. In order to upgrade these results to existence of true tangents, some additional control is required; we replace the finiteness of the Hausdorff $n$-measure with a stronger condition: finiteness of the \emph{$n$-packing content}. See Example \ref{ex:approx-vs-true} for comparison between approximate tangents and true tangents.
	
	The distinction between these two notions of tangent has been examined before in different contexts. We have borrowed the terminology of true tangent from \cite{FTV25} in which the authors show that a certain square function bound allows one to (in our terminology) upgrade an approximate tangent to a true tangent at a point in the common boundary of two domains. Related results in \cite{HV22} and \cite{Vi20} also give characterizations of the existence of a true tangent at a point in different settings using different geometric square functions.
	
\begin{figure}
 \centering
  \includegraphics[width=0.3\textwidth]{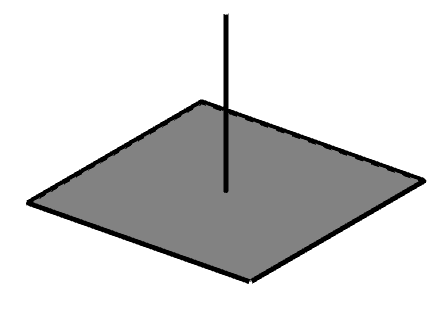}
 \caption{If $E=[0,1]^2\times\{0\} \cup \{(0,0)\}\times[0,1]$, then $E$ has an approximate tangent plane at the origin but not a ``true'' tangent plane. Here $\tilde{E} = \{(0,0)\}\times[0,1]$.}
 \end{figure}	


	While Rademacher's theorem--and its geometric versions--generalize from Lipschitz functions to $p$-Sobolev functions with $p>n$, any attempt for further generalization would be hopeless. The assumption $p>n$ is sharp: in Example \ref{ex:psharp}, we construct for each $n\geq 2$ a function $f\in W^{1,n}(\mathbb{B}^n,\R^{n+1})$ such that on a set of points $x\in f(\mathbb{B}^n)$ of positive $\Haus^n$ measure, there exists a tangent in $\TanAW(f(\mathbb{B}^n),x)$ that is not an $n$-plane.  For maps of lower regularity, the situation is even more complex. In \cite{ShawV}, Shaw and the third named author constructed for each $s>1$, a $(1/s)$-H\"older map $f:[0,1]\to \R^n$ such that $\Haus^s(f([0,1])) >0$ and for $\Haus^s$-almost every $x$ in the image, the space $\TanAW(f([0,1]),x)$ contains infinitely many topologically distinct elements.
	
	Finally, we note that a weaker version of tangents (known as \emph{pseudo-tangents}) exists in the literature. A metric space $T$ is a pseudo-tangent of $X\subset \R^n$ at a point $x\in X$ in the sense of Attouch-Wets (resp.~Gromov-Hausdorff), if there exists a sequence of points $(x_j)$ in $X$ and a sequence of positive numbers $(r_j)$ such that $x_j \to x$, $r_j \to 0$, and spaces $r_j^{-1}(X-x_j)$ (resp.~pointed spaces $(r_j^{-1}X,x_j)$) converge to $T$ in the Attouch-Wets (resp.~Gromov-Hausdorff) topology. It turns out that neither of our main results extends to pseudo-tangents; see Example \ref{ex:pseudot}.
	
	\section{Preliminaries}
	
	Here and for the rest of the paper, when we write $x \lesssim y$, we mean there is a constant $C > 0$ such that $x \leq Cy$. If the constant $C$ depends on parameters $c_1,\dots,c_n$, then we write $x \lesssim_{c_1,\dots,c_n} y$.
	
	Given a metric space $X$ and $r\in(0,\infty)$, we let $rX$ denote the metric space with the metric on $X$ scaled by $r$; that is, $X$ and $rX$ refer to the same point set, but  $d_{rX}(x,y)=rd_X(x,y)$ for all $x,y\in X$. A pointed metric space $(X,x)$ is a metric space $X$ with a distinguished point $x\in X$. We write $f:(X,x)\rightarrow (Z,z)$ if $f$ is a function with domain $X$ and codomain $Z$ such that $f(x)=z$.
	
	Let $X$ be a metric space and let $E\subset X$ be a nonempty set. We say that $E$ is \emph{proper} if every closed ball in the subspace topology on $E$ is compact. Moreover, we denote by $d_X$ the metric of $X$ and by $B_X(x,r):=\{y\in X:d_X(x,y)\leq r\}$ the closed ball with center $x\in X$ and radius $r>0$. We also write $B_X(E,r) := \bigcup_{x\in E}B_X(x,r)$. Finally, if $B=B_X(x,r)$ is a ball, then we write $r(B) = r$. If $X=\R^n$, then we denote balls and closed neighborhoods by $B^n(x,r)$ and $B^n(E,r)$. For the rest of the paper, we denote by $\textbf{0}$ the origin of $\R^n$ (or a Banach space) and write
	\[ \mathbb{B}^n := B^n(\textbf{0},1)\subset \R^n.\]
	
	If $X$ is a metric space and if $s>0$, then define the Hausdorff $s$-measure on the Borel $\sigma$-algebra of $X$ by
	\[ \mathcal{H}^s(A) = \lim_{\delta\to 0} \mathcal{H}^s_{\delta}(A) = \lim_{\delta\to 0} \inf\left\{ \sum_{i=1}^{\infty}(\diam{U_i})^s : \bigcup_{i=1}^{\infty}U_i \supset A, \, \diam{U_i} \leq \delta \right\}.\]
	We similarly define the Hausdorff $s$-content by
	\[\mathcal{H}^s_\infty(A) = \inf\left\{ \sum_{i=1}^{\infty}(\diam{U_i})^s : \bigcup_{i=1}^{\infty}U_i \supset A \right\}. \]
	
	\begin{definition}[Rectifiability]
		A metric space $X$ is \emph{$n$-rectifiable} if there exist countably many Lipschitz maps $f_i:E_i\subset[0,1]^n\rightarrow X$ with $\Haus^n(X\setminus \bigcup_i f_i(E_i))=0$.
	\end{definition}
	
	\begin{definition}[Finite packing content]\label{def:goodpack}
		Given $s>0$, we say that	a metric space $X$ has \emph{finite $s$-packing content} if there exists $M\in (0,+\infty)$ such that if $B_1,B_2,\dots \subset X$ are mutually disjoint balls satisfying $r(B_i) \leq \diam X$, then
		\begin{equation}\label{eq:n-pack}
			\sum_{i}r(B_i)^s <M.
		\end{equation}
	\end{definition}
	
	The motivation and terminology in Definition \ref{def:goodpack} is based on the notion of the \emph{packing measure}.
	
	\begin{definition}[Packing measure]
		Let $X$ be a metric space and let $s>0$. For any $\delta < \diam(X)$, let $\scB_\delta$ be the family consisting of all countable collections of pairwise disjoint balls centered in $X$ with radii less than or equal to $\delta$. Define
		$$ P^s_\delta(X) = \sup_{\mathcal{B}\in\scB_\delta}\sum_{B\in\mathcal{B}}(2r(B))^s.$$
		The \emph{$s$-packing premeasure} of $X$ is given by
		$$P^s(X) = \lim_{\delta\rightarrow0} P^s_\delta(X).$$
		We obtain the \emph{$s$-packing measure} by defining
		$$\mathcal{P}^s(X) = \inf\bigg\{ \sum_{j\in J}P^s(X_j) : X\subset\bigcup_{j\in J}X_j,\ \text{ $J$ countable} \bigg\}.$$
	\end{definition}
	
	The \emph{packing content} is in some ways a counterpart to the Hausdorff content, except we must restrict the radii of balls in the definition of packing content to prevent it from being infinite trivially.
	
	\begin{definition}[Packing content]
		Let $X$ be a separable metric space and let $s>0$. Define the \emph{$s$-packing content} of $X$ to be $P^s_{\diam(X)}(X)$. In the interest of maintaining an analogy with Hausdorff $s$-content $\mathcal{H}^s_{\infty}$, we will use the notation $P^s_\infty(X) = P^s_{\diam(X)}(X)$. 
	\end{definition}
	
	\begin{remark}\label{rem:haus-pack}
		If $X$ is a separable metric space and $s>0$, then
		$$\Haus_\infty^s(X) \leq \Haus^s(X) \leq \mathcal{P}^s(X) \leq P^s(X) \leq P^s_\infty(X).$$
		Moreover,  $X$ has finite $s$-packing content in the sense of Definition \ref{def:goodpack} if and only if $P^s_\infty(X)  < \infty$.
	\end{remark}
	
	\begin{remark}\label{rem:n-pack}
		If $X$ supports a finite and lower $s$-regular measure $\mu$ (i.e. a measure $\mu$ with $\mu(B_X(x,r)) \gtrsim r^s$ for every ball with $r<\diam{X}$), then $X$ has finite $s$-packing content. Indeed, if $B_1,B_2,\dots \subset X$ are mutually disjoint balls of radii at most $\diam{X}$, then
		\[ \sum_{i}r({B_i})^s \lesssim_s \sum_{i}\mu(B_i) \leq \mu(X) < \infty,\]
		implying $P_\infty^s(X) \lesssim \mu(X)$. One can see this as a sort of dual statement to the mass distribution principle; see \cite[\textsection8.7]{Heinonen01}.
	\end{remark}

	\section{Tangent Sets and Metric Tangents}
	
	In this section we discuss tangents in Euclidean and metric spaces.
	
	\subsection{Tangents in metric spaces} Let $X$ be a metric space. Let $A,B,C\subset X$ be nonempty sets. We define the \emph{excess} of $A$ over $B$ by
	\[ \ex(A,B):=\sup_{a\in A}\inf_{b\in B} d_X(a,b) \in[0,\infty].\]
	By convention, we put $\ex(\emptyset,B):=0$, but leave $\ex(A,\emptyset)$ undefined. For every $x\in X$ and $r>0$, we define the \emph{relative Walkup-Wets distance} in $B_X(x,r)$ to be the quantity
	\begin{equation}\label{def:ww}
		\D^{x,r}[A,B]:=r^{-1} \max\{\ex(A\cap B_X(x,r),B), \ex(B\cap B_X(x,r),A)\}\in[0,\infty).
	\end{equation}
	We now record two properties of the relative Walkup-Wets distance from \cite[Lemma 2.2]{BL15}: \emph{Montonicity} and a \emph{weak quasitriangle inequality}.
	\begin{enumerate}
		\item Monotonicity: If $B_X(x,r) \subset B_X(y,s)$, then
		\begin{equation}\label{eq:quasimonotone}
			\text{If $B_X(x,r) \subset B_X(y,s)$, then $\D^{x,r}[A,B] \leq (s/r) \D^{y,s}[A,B]$.}
		\end{equation}
		\item Weak quasitriangle inequality: If $x\in \overline{B}$, then
		\begin{equation}\label{eq:quasitriangle}
			\D^{x,r}[A,C] \leq 2\D^{x,2r}[A,\overline{B}] + 2\D^{x,2r}[\overline{B},C].
		\end{equation}
		
	\end{enumerate}

	We define \emph{Gromov's pointed distance} by $$H^x[A,B]:=\inf\{\epsilon\in(0,1]: \D^{x,1/\epsilon}[A,B]\leq \epsilon^2\},$$ where we interpret $H^x[A,B]=1$ if the set of admissible $\epsilon$ is empty. This is equivalent to the definition in \cite[Section 6]{Gr81}. Unlike the relative Walkup-Wets distance, the Gromov pointed distance is an actual distance.
	
	The next lemma gives a useful relation between the Walkup-Wets distance and the Gromov pointed distance.
	
	\begin{lemma}\label{lem:1}
		Given $A,B\subset X$, $x\in X$, and $\epsilon\in(0,1)$, we have $H^x[A,B] \leq \epsilon$ if and only if $\D^{x,1/\delta}[A,B] \leq \delta^2$ for all $\delta \in (\epsilon,1)$.
	\end{lemma}
	
	\begin{proof}
		Fix $A,B\subset X$, $x\in X$, and $\epsilon\in(0,1)$.
		
		Assume first that $H^x[A,B] \leq \epsilon$. Then by \eqref{eq:quasimonotone}, for any $\theta \in (0,1-\epsilon)$ there exists $t\in [\epsilon,\epsilon+\theta)$ such that
		\[ \D^{x,1/t}[A,B] \leq \frac{1/\epsilon}{1/t}\D^{x,1/\epsilon}[A,B]  \leq  \frac{1/\epsilon}{1/t} \epsilon^2 = \epsilon t < t^2.\]
		If $\delta$ satisfies $0<\delta^{-1} \leq (\epsilon+\theta)^{-1}$, then $\delta^{-1} \leq t^{-1}$ and \eqref{eq:quasimonotone} gives
		\[ \D^{x,1/\delta}[A,B] \leq \frac{1/t}{1/\delta}\D^{x,1/t}[A,B] \leq t\delta \leq \delta^2.\]
		Letting $\theta\to 0$, we conclude that $\D^{x,1/\delta}[A,B] \leq \delta^2$ for all $\delta \in (\epsilon,1)$.
		
		Assume now that $\D^{x,1/\delta}[A,B] \leq \delta^2$ for all $\delta \in (\epsilon,1)$. Then, $H^x[A,B] \leq \delta$ for all $\delta \in (\epsilon,1)$. Letting $\delta \to \epsilon$, we get $H^x[A,B] \leq \epsilon$.
	\end{proof}
	
	
	\begin{definition}[Gromov-Hausdorff distances]\label{def:GHdist}
		For any two pointed metric spaces $(X,x)$ and $(Y,y)$, we define the \emph{pointed Gromov-Hausdorff distance}
		\begin{equation}
			d_{\pGH}((X,x),(Y,y)) := \inf_{\iota_1,\iota_2} H^z[\iota_1(X),\iota_2(Y)]
		\end{equation}
		where the infimum runs over all isometric embeddings $\iota_1:(X,x)\rightarrow (Z,z)$ and $\iota_2:(Y,y)\rightarrow (Z,z)$ into a common pointed metric space $(Z,z)$. It will be convenient to measure this using the following Gromov-Hausdorff variant of the Walkup-Wets distance. For any $r > 0$ and any pointed metric spaces $(X,x)$ and $(Y, y)$, we define
		\begin{equation}
			\D_\GH^{r}[(X,x), (Y,y)] \vcentcolon= \inf_{\iota_1,\iota_2}\D^{z,r}[\iota_1(X), \iota_2(Y)]
		\end{equation}
		where the infimum runs over all isometric embeddings as above. It is straightforward to show that $d_{\pGH}((X,x),(Y,y)) \leq \epsilon$ if and only if $\D_\GH^{1/\epsilon}[(X,x), (Y,y)] \leq \epsilon^2$, and $\D_\GH$ satisfies the following weak form of monotonicity: For any $r < s$,
		\begin{equation}
			\D_\GH^{r}[(X,x), (Y,y)] \leq \frac{s}{r}\D_\GH^s[(X,x),(Y,y)].
		\end{equation}
	\end{definition}
	
	Finally, we give the definition of tangent space that we will use.
	
	\begin{definition}[metric tangents]\label{def:met-tan}
		Let $X$ be a complete metric space and let $x\in X$. We say that a complete pointed metric space $(Y,y)$ is a \emph{metric tangent} or \emph{pointed Gromov-Hausdorff tangent} of $X$ at $x$ if $(r_i^{-1}X,x)\rightarrow (Y,y)$ in the pointed Gromov-Hausdorff topology for some sequence $r_i\rightarrow 0$. We let $\Tan_\GH(X,x)$ denote the set of pointed isometry classes of metric tangents of $X$ at $x$.
	\end{definition}
	
	By Gromov's compactness criterion \cite[Section 6]{Gr81}, any locally compact doubling space admits metric tangents at every point; see also \cite[Theorem 2.11]{Heinonen03}. The following lemma frames phrases the existence of Gromov-Hausdorff tangents in terms of $\D_\GH$.
	
	\begin{lemma}\label{l:dgh-basics}
		We have that $[(Y,y)]\in\TanGH(X,x)$ if and only if there exists a sequence $r_i\rightarrow0$ such that \[\lim_{i\rightarrow\infty}\D^R_\GH[(r_i^{-1}X,x),(Y,y)] = 0 \hbox{\quad for all $R > 0$}.\]
	\end{lemma}
	
	\begin{proof}
		Assume first that $[(Y,y)]\in\TanGH(X,x)$ and fix $R>0$ and $\epsilon < 1/R$. There exists a sequence $(r_i)_i$ and there exists $i_0 \geq 0$ such that $r_i\to0$ and $d_{\pGH}((r_i^{-1}X,x),(Y,y)) \leq \epsilon$ for all $i \geq i_0$. By monotonicity,
		\[\D_\GH^R[(r_i^{-1}X,x),(Y,y)]  \leq \frac{1/\epsilon}{R}\D_\GH^{1/\epsilon}[(r_i^{-1}X,x),(Y,y)] \leq \frac{1/\epsilon}{R}\epsilon^2 \leq \frac{\epsilon}{R}\]
		for all $i \geq i_0$. Letting $\epsilon\rightarrow0$ shows $\lim_{i\rightarrow\infty}D_\GH^R[(r_i^{-1}X,x),(Y,y)]  = 0$. For the converse, fix a sequence $r_i\to 0$ such that $\lim_{i\to\infty}\D_\GH^{R}[(r_i^{-1}X,x),(Y,y)] =0$ for all $R>0$. Fix also $\epsilon > 0$ and let $i_0$ be such that $\D_\GH^{1/\epsilon}[(r_i^{-1}X,x),(Y,y)] < \epsilon^2$ for all $i \geq i_0$. Then $d_{\pGH}((r_i^{-1}X,x),(Y,y)) < \epsilon$ for all $i \geq i_0$, showing
		\[ \lim_{i\rightarrow\infty}d_{\pGH}((r_i^{-1}X,x),(Y,y)) = 0. \qedhere\]		
	\end{proof}
	We will also need to understand the relationship between the pointed Gromov-Hausdorff distance between spaces and the existence of different kinds of ``approximate isometries" between spaces. We will use the notion of \emph{$\epsilon$-isometry} from \cite{Ba22}.
	\begin{definition}[$\epsilon$-isometry]
		Let $(X,x)$, $(Y,y)$ be pointed metric spaces, and let $\epsilon > 0$. We say that a Borel map $f:(X,x)\rightarrow (Y,y)$ is an $\epsilon$-isometry if $f(x) = y$ and both
		\begin{equation*}
			\sup_{z,z' \in B_X(x,1/\epsilon)}|d_Y(f(z),f(z')) - d_X(z,z')| \leq \epsilon,
		\end{equation*}
		and
		\begin{equation*}
			B_Y(y,1/\epsilon - \epsilon) \subset B_Y(f(B_X(x,1/\epsilon)),\epsilon).
		\end{equation*}
	\end{definition}
	
	\begin{lemma}[{\cite[Lemma 2.24]{Ba22}}]\label{l:GH-isom}
		Let $0 < \epsilon < 1/2$. Let $(X,x)$ and $(Y,y)$ be pointed metric spaces. The following statements hold:
		\begin{enumerate}
			\item \label{item:pGH-implies-eps-isom} If $X$ and $Y$ are separable and $d_{\pGH}((X,x),(Y,y)) < \epsilon$, then there exists a pointed Borel map $f:X\cap B_X(x,1/\epsilon)\rightarrow (Y,y)$ such that \[|d_Y(f(x'),f(x'')) - d_X(x',x'')| \leq 2\epsilon\] for all $x',x''\in B_X(x,1/\epsilon)$ and \[B_Y(y,1/\epsilon-\epsilon)\subset B_Y(f(B_X(x,1/\epsilon)),2\epsilon).\] In particular, there exists a $2\epsilon$-isometry between $(X,x)$ and $(Y,y)$.
			
			\item \label{item:eps-isom-implies-pGH} If there exists an $\epsilon$-isometry from $(X,x)$ to $(Y,y)$, then
			\[ d_\pGH((X,x),(Y,y)) \leq 2\epsilon.\]
		\end{enumerate}
	\end{lemma}
	\begin{lemma}[excess lower bound]\label{l:mt-ex-lb}
		Suppose that $(Z,z)$ is a pointed metric space. Let $\td{X},X,Y\subset Z$ with $\td{X}\subset X$, let $z\in \td{X}\cap Y$, and let $\epsilon , r > 0$. If $\td{X}$ is closed and
		\begin{equation*}
			\D^{z,r/4}[X,Y] > \epsilon \text{\quad and \quad} \D^{z,r}[\td{X},Y] < \frac{\epsilon}{8},
		\end{equation*}
		then
		\begin{equation*}
			\ex(X\cap B_Z(z,r/2),\td{X}) \geq \frac{\epsilon}{8}r.
		\end{equation*}
	\end{lemma}
	\begin{proof}
		The fact that $\td{X}\subset X$ implies $\ex(\td{X}\cap B_Z(z,r),X) = 0$. Using the weak quasitriangle inequality \eqref{eq:quasitriangle}, the fact that $x\in \td{X}\cap Y$, and monotonicity \eqref{eq:quasimonotone}, we get
		\begin{align*}
			\epsilon < \D^{z,r/4}[X, Y] &\leq 2\D^{z,r/2}[X,\td{X}] + 2\D^{z,r/2}[\td{X},Y] \\
			&\leq 4\frac{\ex(X\cap B_Z(z,r/2),\td{X})}{r} + 4\D^{z,r}[\td{X},Y]\\
			& \leq \frac{4}{r}\ex(X\cap B_Z(z,r/2), \td{X}) + \frac{\epsilon}{2}.
		\end{align*}
		Rearranging this inequality gives $\ex(X\cap B_Z(z,r/2),\td{X}) \geq (\epsilon/8) r$.
	\end{proof}
	
	\subsection{Attouch-Wets tangents}\label{sec:euc_tang}
	Because the Gromov-Hausdorff distance does not distinguish between rotating tangent planes in Banach spaces, we require different notions of tangents and some modified arguments for the Euclidean versions of our results.
	
	The notion of distance we use in defining tangents of subsets of Euclidean spaces actually induces the \textit{Attouch-Wets} topology on closed sets. Given $n\in\N$, let $\mathfrak{C}(\R^n)$ be the set of nonempty closed subsets of $\R^n$; let $\mathfrak{C}(\R^n,{\bf 0})$ be the collection of nonempty closed subsets of $\R^n$ containing the origin ${\bf 0}$. We consider both of these spaces equipped with the \emph{Attouch-Wets} topology, which is defined in \cite[Definition 3.1.2]{Beer}. The following lemma gives a useful characterization of this convergence and tells us that the subfamily of closed sets containing the origin is sequentially compact in this topology.
	
	\begin{lemma}[{\cite[Theorem 3.1.7]{Beer}, \cite[Lemma 8.2]{DS97}}]\label{lem:AW-char}
		There exists a metrizable topology on the set $\mathfrak{C}(\R^n)$ in which a sequence of sets $(X_i)_{i=1}^{\infty}\subset \mathfrak{C}(\R^n)$ converges to a set $X\in\mathfrak{C}(\R^n)$ if and only if
		\[ \lim_{m\to\infty} \D^{{\bf 0}, r}(X_m,X) = 0, \quad\text{for every $r>0$}.\]
		Moreover, the subcollection $\mathfrak{C}(\R^n,{\bf 0})$ is sequentially compact; that is, for any sequence $(X_i)_{i=1}^{\infty}\subset \mathfrak{C}(\R^n,{\bf 0})$ there exists a subsequence $(X_{i_j})_{j=1}^{\infty}$ and a set $X \in \mathfrak{C}(\R^n,{\bf 0})$ such that $(X_{i_j})_{j=1}^{\infty}$ converges to $X$.
		\end{lemma}
	
	
	Following \cite{BL15}, we can now define the notion of tangent with respect to this topology.
	
	\begin{definition}[Attouch-Wets Tangent]
		Let $E\in\mathfrak{C}(\R^n)$, $T\in\mathfrak{C}(\R^n,{\bf 0})$, and let $x\in E$. We say that $T$ is a \textit{tangent set} or an \textit{Attouch-Wets tangent} to $E$ at $x$ if and only if there exists a sequence $r_i\rightarrow 0$ such that
		\begin{equation*}
			r_i^{-1}(E-x)\rightarrow T \hbox{ in the Attouch-Wets topology}.
		\end{equation*}
		We denote the set of all Attouch-Wets tangents to $E$ at $x$ by $\Tan_\AW(E,x)$.
	\end{definition}
	
	By Lemma \ref{lem:AW-char}, every closed set $X\subset\R^n$ has Attouch-Wets tangents at every point in the set.
	
	\begin{remark} One could attempt to extend this tangent notion by replacing $\R^n$ with an abstract Banach space $V$. However, unlike metric tangents,  the existence of Attouch-Wets tangents in this more general setting is not guaranteed, even if the set $E$ is doubling and compact. For instance, let $E = \{\bz\}\cup\{2^{-n}\textbf{e}_n : n\in\N\} \subset \ell_{\infty}$ where $\{\textbf{e}_n\}_n$ is the standard basis of $\ell_{\infty}$. It is easy to see that $E$ is doubling and compact. Fix now a sequence $(r_j)_j$ of positive numbers with $r_j\to 0$. We claim that $E$ has no Attouch-Wets tangent at $\bz$. To see that, fix $i,j \in \N$ with $r_j<1$ and $r_i < (1/4)r_j$, and let $n_i,n_j \in \N$ such that $r_i \in (2^{-n_i},2^{-n_i+1}]$ and $r_j \in (2^{-n_j},2^{-n_j+1}]$. Note that $\frac{2^{-n_j}\textbf{e}_{n_j}}{r_j} \in (r_j^{-1}E)\cap B_{\ell_{\infty}}(\bz,1)$, since
		\[\frac{1}{2} \leq \bigg|\frac{2^{-n_j}\textbf{e}_{n_j}}{r_j}\bigg| \leq 1.\]
		Because $(r_i^{-1}E)\cap B_{\ell_{\infty}}(\bz,1)$ is orthogonal to $\textbf{e}_{n_j}$, we know
		\[ \ex((r_j^{-1}E)\cap B_{\ell_{\infty}}(\bz,1),r_i^{-1}E) \geq \inf_{x\in (r_i^{-1}E) \cap B_{\ell_{\infty}}(\bz,1)}\dist(\textbf{e}_j,x) \geq \frac{1}{2}.\]
		The latter, along with the triangle inequality for excess \cite[Section 2]{BL15}, implies that there exists no subsequence of $(r_k^{-1}E)_{k\in\N}$ that converges in the Attouch-Wets topology.
	\end{remark}
	
	In the following lemma, which is the analogue of Lemma \ref{l:dgh-basics}, we record some basic relationships between Attouch-Wets convergence, the Walkup-Wets distance, and Gromov's pointed distance.
	
	\begin{lemma}\label{lem:AW-ph}
		Let $(A_i)_{i\in\N}\subset \mathfrak{C}(\R^d)$ and $A\in \mathfrak{C}(\R^d)$. The following are equivalent:
		\begin{enumerate}
			\item $A_i \to A$ in the Attouch-Wets topology,
			\item there exists $x \in \R^d$ such that $\lim_{i\to \infty}\D^{x,r}[A_i,A]=0$ for all $r>0$,
			\item $\lim_{i\to\infty}\D^{x,r}[A_i,A]=0$ for all $x\in\R^d$ and all $r>0$,
			\item there exists $x \in \R^d$ such that $\lim_{i\to \infty}H^{x}[A_i,A]=0$,
			\item $\lim_{i\to \infty}H^{x}[A_i,A]=0$ for all $x\in \R^d$.
		\end{enumerate}
	\end{lemma}
	
	\begin{proof}
		By definition, (1) implies (2), and (3) implies (1).
		
		To show that (2) implies (3), fix $x,y \in \R^d$ and $r>0$, and assume that $\lim_{i\to\infty}\D^{x,R}[A_i,A]=0$ for all $R>0$. By \eqref{eq:quasimonotone}, we have that
		\[ \D^{y,r}[A_i,A] \leq \frac{|x-y|+r}{r}\D^{x,r+|x-y|}[A_i,A] \xrightarrow{i\to \infty}0.\]
		
		Clearly, (5) implies (4). To show that (4) implies (2), fix $x\in \R^d$ and $r>0$, and assume that $\lim_{i\to \infty}H^x[A_i,A]=0$. Let $\epsilon \in (0,1/r)$ and let $i_0 \in\N$ such that $H^x[A_i,A] \leq \epsilon$ for all $i\geq i_0$. By Lemma \ref{lem:1} and by \eqref{eq:quasimonotone}, for all $i\geq i_0$
		\[ \D^{x,r}[A_i,A] \leq \frac{1/\epsilon}{r}\D^{x,1/\epsilon}[A_i,A] \leq \frac{1/\epsilon}{r}\epsilon^2 = \epsilon/ r.\]
		Letting $\epsilon \to 0$, we get $\lim_{i\to \infty}\D^{x,r}[A_i,A]=0$.
		
		It remains to show that (3) implies (5). To this end, fix $x\in \R^d$ and let $\epsilon>0$. There exists $i_0 \in \N$ such that for all $i \geq i_0$, $\D^{x,1/\epsilon}[A_i,A] < \epsilon^2$. Then, by definition, $H^x[A_i,A] \leq \epsilon$ for all $i\geq i_0$. Therefore, $\lim_{i\to \infty} H^x[A_i,A] =0$.
	\end{proof}

	\subsection{Conical tangents}
	In this subsection, we derive an additional lemma on conical metric tangents. First, recall the definition of a Euclidean cone.
	\begin{definition}[Euclidean cones]
		We say that $T\subset\R^d$ is a \emph{cone} if $\lambda T = T$ for every $\lambda > 0$.
	\end{definition}
	While Euclidean cones are rather straightforward to work with, it will take some extra work to make similar claims about \emph{metric cones}. The motivating example of a metric cone is a normed space, but several other types of metric cones appear in metric geometry (See \cite{BBI01} Section 3.6.2 and Section 8.2).
	
		
	\begin{definition}[metric cones]\label{def:metric-cones}
		Let $Z$ be a metric space. Define a quotient space
		\[ C(Z) = (Z\times[0,\infty)) / (Z\times\{0\}) \text{\quad and \quad} c = Z\times\{0\}.\]
		We say that the pointed metric space $(C(Z),c)$ is a \emph{metric cone} if the following hold:
		\begin{enumerate}
			\item \label{item:euc-rays} For any $z\in Z$, the space $\{(z,t) : t > 0\}\subset C(Z)$ is isometric to the Euclidean ray $(0,\infty)$.
			\item \label{item:similarities} For any $\lambda > 0$, there exists a pointed bijection $\psi:(C(Z),c)\rightarrow (C(Z),c)$ such that
			\[d_{C(Z)}(\psi(x),\psi(y)) = \lambda d_{C(Z)}(x,y) \hbox{\quad for all $x,y\in C(Z)$}\]
			We call $\psi$ a $\lambda$-\emph{similarity}.
		\end{enumerate}
	\end{definition}
	We will use the following type of approximate isometry from \cite{BHS23} between spaces to study metric cones.
	\begin{definition}[Gromov-Hausdorff approximations]
		Let $(X,x)$ and $(Y,y)$ be pointed metric spaces and $\epsilon > 0$. A Borel mapping $\phi:X\rightarrow Y$ is a \emph{pointed $\epsilon$-Gromov-Hausdorff approximation} (or $\epsilon$-pGHA) if
		\[\sup_{w,v\in X}|d_Y(\phi(w),\phi(v)) - d_X(w,v)| \leq \epsilon \quad\text{and}\quad \sup_{u\in Y}\dist_Y(u,\phi(X)) \leq \epsilon.\]
		We use this to define
		\begin{align*}
			\xi_{X,Y}(x,y,r) \vcentcolon= \inf\{\epsilon > 0: \text{there exists an }\epsilon r&\text{-pGHA}\\
			&\phi:(B_X(x,r),x)\rightarrow (B_Y(y,r),y)\}.
		\end{align*}
	\end{definition}
	Pointed Gromov-Hausdorff approximations are similar to $\epsilon$-isometries, but they are more suited to quantifying how close two metric spaces are inside fixed balls using $\xi$. We will use the following lemma later to pass between bounds on the pointed Gromov-Hausdorff distance of $(r^{-1}X,x)$ to a normed space and the existence of good pGHAs from $B_X(x,r)$ to a normed ball of radius $r$ for $r\rightarrow 0$.
	\begin{lemma}[conical metric tangents]\label{l:mt-cones-and-pGHAs}
		Suppose $(X,x)$ and $(Y,y)$ are separable pointed metric spaces and $(Y,y)$ is a metric cone. Let $\epsilon, r > 0$.
		\begin{enumerate}
			\item\label{item:DGH-implies-xi} If $\D_\GH^{1/\epsilon}[(r^{-1}X,x),(Y,y)] < \epsilon^2$, then $\xi_{X,Y}(x,y,r/\epsilon) < 3\epsilon^2$.
			\item\label{item:xi-implies-DGH} If $\xi_{X,Y}(x,y,r/\epsilon) < \epsilon^2$, then $\sD^{1/2\epsilon}_{\GH}[( r^{-1}X,x),(Y,y)] < 4\epsilon^2$.
		\end{enumerate}
		In particular, $\lim_{r\rightarrow0}\xi_{X,Y}(x,y,r) = 0$ if and only if
		$\lim_{r\rightarrow0} \D_\GH^{R}[(r^{-1}X,x),(Y,y)] = 0$ for every $R > 0$.
	\end{lemma}
	
	\begin{proof}
		First, suppose $\D_\GH^{1/\epsilon}[(r^{-1}X,x),(Y,y)] < \epsilon^2$. Then $d_{\pGH}((r^{-1}X,x), (Y,y)) < \epsilon$, and we can apply Lemma \ref{l:GH-isom}(1) to get a mapping $f:B_{r^{-1}X}(x,1/\epsilon)\rightarrow Y$ satisfying the properties listed in the lemma. Now, $B_{r^{-1}X}(x,1/\epsilon) = B_X(x,r/\epsilon)$ as sets, and Definition \ref{def:metric-cones}(2) implies there exists an $r$-similarity $\psi:B_Y(y,1/\epsilon)\rightarrow B_Y(y,r/\epsilon)$. We claim that $\psi\circ f:B_X(x,r/\epsilon)\rightarrow B_Y(y,r/\epsilon)$ is a $3\epsilon^2 r$-pGHA. Indeed, for any $x',x''\in B_X(x,r/\epsilon)$,
		\begin{align*}
			|d_Y(\psi(f(x')),\psi(f(x''))) - d_X(x',x'')| &= r|d_Y(f(x'),f(x'')) -  r^{-1}d_X(x',x'')|\\
			&\leq 2\epsilon r.
		\end{align*}
		Now, Definition \ref{def:metric-cones}(1) implies that for any $\tilde{u}\in B_Y(y,1/\epsilon)$ there exists $\tilde{w}\in B_Y(y,1/\epsilon-\epsilon)$ with $|\tilde{u}-\tilde{w}| = \epsilon$. This means
		\begin{align*}
			\sup_{u\in B_Y(y,r/\epsilon)}\dist(u,\psi(f(B_X(x,r))) &= r\sup_{\td{u}\in B_Y(y,1/\epsilon)}\dist(\td{u}, f(B_{(\epsilon r)^{-1}X}(x,1/\epsilon))) \\
			&\leq \epsilon r + r\sup_{\td{w}\in B_Y(y,1/\epsilon-\epsilon)}\!\!\!\dist(\td{w}, f(B_{(\epsilon r)^{-1}X}(x,1/\epsilon)))\\
			&\leq 3\epsilon r.
		\end{align*}
		Hence $\xi_{X,Y}(x,y,r/\epsilon) \leq 3\epsilon^2$.
		
		Now, suppose that $\xi_{X,Y}(x,y,r/\epsilon) < \epsilon^2$. Let $\phi:B_X(x,r/\epsilon)\rightarrow B_Y(y,r/\epsilon)$ be an $\epsilon^2 r/\epsilon = \epsilon r$-pGHA, and let $\psi:B_Y(y,r/\epsilon)\rightarrow B_Y(y,1/\epsilon)$ be an $r^{-1}$-similarity. We claim that $\psi\circ\phi:B_{ r^{-1}X}(x,1/\epsilon)\rightarrow B_Y(y,1/\epsilon)$ is an $\epsilon$-pGHA. First, let $z,z'\in B_{r^{-1}X}(x,1/\epsilon) = B_X(x,r/\epsilon)$. Then,
		\begin{equation*}
			|d_Y(\psi\circ\phi(z),\psi\circ\phi(z')) -  r^{-1}d_X(z,z')| = r^{-1}|d_Y(\phi(z), \phi(z')) - d_X(z,z')|\leq \epsilon
		\end{equation*}
		Second, let $y'\in B_{Y}(y,1/\epsilon)$ and observe that
		\begin{equation*}
			\dist_Y(y',\psi\circ\phi(B_{r^{-1}X}(x,1/\epsilon))) = r^{-1}\dist_Y(\psi^{-1}(y'),\phi(B_X(x,r/\epsilon))) \leq \epsilon.
		\end{equation*}
		This verifies that $\psi\circ\phi$ is an $\epsilon$-pGHA. In particular, $\psi\circ\phi$ is an $\epsilon$-isometry. Using Lemma \ref{l:GH-isom}(2), we get
		\[ \D^{1/2\epsilon}_\GH[(r^{-1}X,x),(Y,y)] \leq 4\epsilon^2.\]
		This verifies the first two statements of the lemma. The final statement of the Lemma follows from these statements using the monotonicity of $\D_\GH$ for the forward direction.
	\end{proof}
	
	\subsection{Gromov-Hausdorff versus Attouch-Wets tangents}
	We end this section with a theorem that gives the precise relationship between the Gromov-Hausdorff tangents and the Attouch-Wets tangents to subsets of Euclidean spaces.

	\begin{theorem}\label{t:two-tangents}
		Let $E\in\mathfrak{C}(\R^n)$ and let $x\in E$.
		\begin{enumerate}
			\item If $T\in\Tan_{\AW}(E,x)$, then $(T,{\bf 0})\in\Tan_{\GH}(E,x)$.
			\item If $[(Y,y)]\in\Tan_{\GH}(E,x)$, then there exists $T\in\Tan_{\AW}(E,x)$ such that there exists a pointed isometry between $(T,{\bf 0})$ and $(Y,y)$.
		\end{enumerate}
	\end{theorem}
	
	\begin{proof}
		First, suppose that $T\in\Tan_{\AW}(E,x)$.
		The mappings $\iota_1:E\rightarrow V$ and $\iota_2:T\rightarrow V$ given by $\iota_1(y) = \frac{y-x}{r}$ and $\iota_2(z) = z$ give isometric embeddings of $(E,r^{-1}|\cdot|)$ and $(T,|\cdot|)$ into $\R^n$ with $\iota_1(x) = {\bf 0}$ and $\iota_2({\bf 0}) = {\bf 0}$. This means we can compute
		\begin{equation}\label{e:ph-of-awt}
			d_{\pGH}((E,r^{-1}|\cdot|,x),(T,|\cdot|,{\bf 0})) \leq H^{\bf 0}(\iota_1(E), \iota_2(T)) = H^{\bf 0}\left(r^{-1}(E-x), T\right).
		\end{equation}
		Since $T\in\Tan_{\AW}(E,x)$, Lemma \ref{lem:AW-ph} implies that there exists a sequence $r_i\rightarrow 0$ such that the right hand side of \eqref{e:ph-of-awt} with $r$ replaced by $r_i$ goes to zero as $i\rightarrow\infty$. Then $r_i$ is such that $\lim_{i\rightarrow\infty}d_{\pGH}((E,r_i^{-1}|\cdot|,x),(T,|\cdot|,{\bf 0}))= 0$, showing that $[(T,{\bf 0})]\in\Tan_{\GH}(E,x)$.
		
		Now, let $[(Y,y)]\in\Tan_{\GH}(E,x)$. By Lemma \ref{l:GH-isom}(1), there exists $r_i\rightarrow 0$ and $2^{-i}$-isometries $\iota_i: (Y,y) \rightarrow (r_i^{-1}(E-x),\bz)$. Set $E_i = r_i^{-1}(E-x)$. By Lemma \ref{lem:AW-char}, we can pass to a subsequence and further assume that $(E_i)_{i\in\N}$ converge in the Attouch-Wets topology in $\R^n$ to a limit set $T\in\mathfrak{C}(\R^n,{\bf 0})$. Therefore, $T\in\Tan_{\AW}(E,x)$. Note that $Y$ is separable since it is doubling (see \cite[Proposition 6.1.5]{MT10}) and that for any $x\in Y$, notice that $|\iota_i(x) - \bz| = |\iota_i(x) - \iota_i(y)| \leq d_Y(x,y) + 1$ for all $i\in\N$ so that
		$$ \overline{\{ \iota_i(x) \}_{i\in\N}} \subset B^n(\bz, d_Y(x,y)+1). $$
		The above set is compact and, following the proof of the Arzel\'a-Ascoli Theorem, we obtain  a pointed isometry between $(Y,y)$ and $(T,\bz)$.
	\end{proof}
	
\section{Sobolev maps into metric spaces}\label{sec:Sobolev}
In this section we discuss Sobolev maps in metric spaces. We start by introducing a class of metric space-valued Sobolev spaces. We will only present the specific cases of definitions and results that we need in this paper. For more details and more general statements, we suggest the interested reader consult \cite[Section 7]{HKST15} and \cite{HKST01}. For the rest of this section, we fix a separable metric space $X$.
By the Kuratowski Embedding Theorem,  $X$ isometrically embeds into a Banach space $V$ (e.g. $V=\ell_\infty$), which we fix for the rest of the paper. For convenience, we assume for the rest that $X\subset V$.

\begin{definition}[Metric space-valued Sobolev maps]\label{def:metricSobolev}
	Let $1 \leq p \leq \infty$ and let $\Omega\subset\R^n$ be a Borel set. We say that $u$ is in the \textit{Newtonian-Sobolev space} $\in N^{1,p}(\Omega,X)$ of order $p$ on $\Omega$ with target $X$, if $u\in L^p(\Omega,V)$ and if there exists a Borel function $\rho:\Omega\rightarrow[0,\infty]$ so that $\rho\in L^p(\Omega)$ and
	\begin{equation}\label{e:up-grad}
		d_X(u(\gamma(a)) , u(\gamma(b))) \leq \int_\gamma\rho \, ds
	\end{equation}
	for $p$-a.e. rectifiable curve $\gamma:[a,b]\rightarrow \Omega$.
	
	Here the term``$p$-almost every curve'' refers to the $p$-modulus of curve families; see \cite[Section 5]{HKST15} for $p\in (1,\infty)$ and \cite[Definition 5.5]{Durand} for $p=\infty$.
\end{definition}

In similar fashion we define the Newtonian-Sobolev space $N^{1,p}(\Omega,V)$ replacing $X$ by its superset $V$. Clearly,
\[ N^{1,p}(\Omega,X) \subset N^{1,p}(\Omega,V).\]
It is known that

The following theorem is our main goal for the rest of this section.

\begin{theorem}\label{t:sob-low-reg}
	Let $n < p \leq \infty$. If $f\in N^{1,p}( \mathbb{B}^n,X)$ is continuous, then $f( \mathbb{B}^n)$ is $n$-rectifiable and has finite $n$-packing content.
\end{theorem}

We start by giving a simple proof of Theorem \ref{t:sob-low-reg} in the special case of Lipschitz maps which correspond to the case $p=\infty$ \cite[Corollary 5.24]{Durand}.

\begin{remark}\label{l:pushfwd-lr}
	Let $f: \mathbb{B}^n\rightarrow X$ be $L$-Lipschitz. By definition, $f( \mathbb{B}^n)$ is $n$-rectifiable. Define the push-forward measure $\nu_f = f_*(\scL^n)$ by
	\[\nu_f(A) = \scL^n(f^{-1}(A)) \quad \text{for any Borel $A\subset X$.}\]
	We claim that $\nu_f$ is lower $n$-regular. Indeed, let $y = f(x) \in f(\mathbb{B}^n)$ and $r > 0$. Then $f(B^n(x,r/L)) \subset B_X(f(x), r)$ so that
	\[ \nu_f(B_X(y,r)) \geq \scL^n(B^n(x,r/L)\cap \mathbb{B}^n) \gtrsim_nL^{-n}r^n.\]
	By Remark \ref{rem:n-pack}, $f( \mathbb{B}^n)$ has finite $n$-packing content.
\end{remark}

When $p<\infty$, it will be convenient for us to work with two slightly different notions of Sobolev spaces. The defining property of Reshetnyak-Sobolev spaces in the following proposition will be useful in constructing a lower $n$-regular measure on Sobolev images.

\begin{proposition}\label{t:sob-equiv}
	Let $1<p<\infty$ and $f :  \mathbb{B}^n \to V$ be a Borel function with $f( \mathbb{B}^n)$ being separable. The following are equivalent:
	\begin{enumerate}
		\item $f$ is in the \emph{Newtonian-Sobolev} space $N^{1,p}( \mathbb{B}^n,V)$;
		\item\label{def:resh-sob} $f$ is in the \emph{Reshetnyak-Sobolev} space $W^{1,p}( \mathbb{B}^n,V)$: there exists $g\in L^p(\mathbb{B}^n)$ with $g\geq 0$ such that for any 1-Lipschitz function $\varphi:V\rightarrow\R$, $\varphi\circ f\in W^{1,p}( \mathbb{B}^n,\R)$, and $|\nabla (\varphi\circ f)| \leq g$ holds $\scL^n$-a.e. in $ \mathbb{B}^n$;
		\item $f$ is in the \emph{Haj{\l}asz-Sobolev space}  $M^{1,p}( \mathbb{B}^n,V)$: there exists $g\in L^p( \mathbb{B}^n)$ with $g \geq 0$ such that for all $x,y\in  \mathbb{B}^n$,
		\begin{equation*}
			d_X(f(x),f(y)) \leq |x-y||g(x) + g(y)|.
		\end{equation*}
	\end{enumerate}
\end{proposition}

\begin{proof}
	Since $ \mathbb{B}^n$ supports Poincar{\' e} inequalities and $\scL^n$ is doubling, the equivalence of (1) and (2) follows from \cite[Corollary 10.2.9]{HKST15}.
	
	Now we prove the equivalence of (2) and (3). Since $f( \mathbb{B}^n)$ is separable, property (2') of Theorem 3.17 in \cite{HKST01} holds. But $N^{1,p}( \mathbb{B}^n, \R) = W^{1,p}( \mathbb{B}^n,\R)$ and the minimal $p$-weak upper gradient of any Newtonian-Sobolev function agrees with its classical gradient. Hence, property (2') of Theorem 3.17 in \cite{HKST01} is equivalent to the defining property of $W^{1,p}( \mathbb{B}^n,V)$.
\end{proof}

Under the restriction $p > n$, Newtonian-Sobolev mappings satisfy Lusin's condition N. The following theorem is a special case of the more general result in \cite{LZ24}.

\begin{theorem}[{\cite[Theorem 1.4]{LZ24}}]\label{t:LCN}
	If $p > n$, $\Omega\subset\R^n$ is open, and $f\in N^{1,p}(\Omega,X)$ is continuous, then $f$ satisfies Lusin's condition N: if $E\subset \Omega$ with $\scL^n(E) = 0$, then $\Haus^n(f(E)) = 0$.
\end{theorem}

Next, we recall the definition of $p$-capacity.

\begin{definition}[$p$-capacity]
	Let $K$ be a compact set, let $\Omega$ be an open set with $K\subset\Omega\subset \R^n$ and let $1<p<\infty$. We define $\pcap_p(K,\Omega)$, the $p$-capacity of $K$ relative to $\Omega$, by
	\begin{equation*}
		\pcap_p(K,\Omega) = \inf\left\{ \int_\Omega |\nabla u|^p dx : u\in\scR(K,\Omega) \right\},
	\end{equation*}
	where
	\begin{equation*}
		\scR(K,\Omega) = \{u\in C^\infty_0(\Omega) : u\geq 1\ \text{ on } K\}.
	\end{equation*}
\end{definition}

\begin{lemma}[{\cite[\textsection2.2.3 Corollary 2]{Ma85}}]\label{l:cap-lower-bound}
	Let $K$ be a compact set such that $K\subset \Omega \subset \R^n$ and let $1 < p < \infty,\ p \not= n$. Then
	\begin{equation*}
		\pcap_p(K,\Omega) \geq \omega_n^{p/n}n^{(n-p)/n}\left|\frac{p-n}{p-1}\right|^{p-1}\left|\scL^n(\Omega)^{\frac{p-n}{n(p-1)}} - \scL^n(K)^{\frac{p-n}{n(p-1)}}\right|^{1-p}
	\end{equation*}
\end{lemma}

\begin{remark}
	Let $x\in\Omega$ and $p > n$. Lemma \ref{l:cap-lower-bound} implies
	\begin{equation}\label{e:cap-leb}
		\pcap_p(\{x\},\Omega) \gtrsim_{p,n}\left|\scL^n(\Omega)^{\frac{p-n}{n(p-1)}}\right|^{1-p} = \scL^n(\Omega)^{-\frac{p-n}{n}}.
	\end{equation}
	Equation \eqref{e:cap-leb} is the consequence of Lemma \ref{l:cap-lower-bound} that we will use in our proof.
\end{remark}

As a technical convenience, we will also use the following extension result. The statement is much less general than that proven by the authors; they allow metric space domains and show how to extend off of subsets satisfying much more general conditions than satisfied by a ball.

\begin{lemma}\label{t:sob-ext}
	For every $1 \leq p < \infty$ there exists a a bounded linear extension operator $E_V:M^{1,p}(\mathbb{B}^n,V)\rightarrow M^{1,p}(\R^n,V)$ such that $\|E_V\| \leq C(p)$. Moreover, for any $u\in M^{1,p}(\mathbb{B}^n,V)$, we can assume $E_V(u)=\bz$ on $\R^n \setminus B^n(\bz,2)$.
\end{lemma}

\begin{proof}
	The first claim follows from \cite[Theorem 5.1]{GBIZ23}. The second claim follows from the proof of \cite[Corollary 5.9]{GBIZ23} in which the extension to the entire space is defined by multiplying an already produced extension by a cutoff function.
\end{proof}

We are now ready to prove Theorem \ref{t:sob-low-reg}.

\begin{proof}[{Proof of Theorem \ref{t:sob-low-reg}}]
	Since $f\in N^{1,p}( \mathbb{B}^n,X)$, Proposition \ref{t:sob-equiv} implies that $f\in M^{1,p}( \mathbb{B}^n,X)$ so that there exists a non-negative $g\in L^p( \mathbb{B}^n)$ such that Theorem \ref{t:sob-equiv}(3) holds. Let
	\begin{align*}
		A &= \{x\in  \mathbb{B}^n : g(x) < 1\},\quad A_j = \{x\in  \mathbb{B}^n : 2^{j} \leq g(x) < 2^{j+1}\}.
	\end{align*}
	Since $g\in L^p$, we have that $ \mathbb{B}^n = \bigcup_{j=0}^\infty A_j \cup A \cup N$ where $\Haus^n(N) = 0$ and all sets constituting the unions are disjoint. This means $f( \mathbb{B}^n) \subset \bigcup_{i=1}^\infty f(A_i) \cup f(A')\cup f(N)$ where $\Haus^n(f(N)) = 0$ because $f$ satisfies Lusin's condition N by Theorem \ref{t:LCN}. Since $f|_{A_i}$ is $2^{j+1}$-Lipschitz and $f|_{A}$ is 2-Lipschitz, this shows that $f(\mathbb{B}^n)$ is $n$-rectifiable.
	
	Assume now that $n<p<\infty$. It suffices to prove (2) for $f\in N^{1,p}( \mathbb{B}^n,X)\subset N^{1,p}( \mathbb{B}^n,V)$ such that $f(0) = \bz$, where $\bz$ is the origin in $V$.
	Since $f( \mathbb{B}^n)\subset X$, we have that $f( \mathbb{B}^n)$ is separable and, by Proposition \ref{t:sob-equiv}, $f\in M^{1,p}( \mathbb{B}^n,V)$. We can apply Lemma \ref{t:sob-ext} to extend $f$ to a function $\td{f} = E_V(f)\in M^{1,p}(\R^n,V)$ such that $\td{f}=\bz$ on $\R^n\setminus B^n(\bz,2)$. We can assume $\tilde{f}$ is continuous, agrees with $f$ on $ \mathbb{B}^n$, and $\tilde{f}\in M^{1,p}(B^n(\bz,2),V)$ with $\|\tilde{f}\|_{M^{1,p}} \lesssim \|f\|_{M^{1,p}}$. Let $E = \tilde{f}(B^n(\bz,2))$. Since $E$ is compact, $\diam(E) < \infty$ and $D = \max\{1, \diam(E)\} < \infty$. By Proposition \ref{t:sob-equiv}, $\tilde{f}\in W^{1,p}(B^n(\bz,2),V)$, so there exists $g\in L^p(B^n(\bz,2))$ satisfying Proposition \ref{t:sob-equiv}(2).
	
	We claim that the measure $\nu$ given by
	\begin{equation*}
		\nu(A) = D^n\delta_\bz(A) + \tilde{f}_*[(1+g^p)dx](A)
	\end{equation*}
	is finite and lower $n$-regular. Assuming the claim, Remark \ref{rem:n-pack} implies that $E$ has finite $n$-packing content, so $f(\mathbb{B}^n)$ has finite $n$-packing content as a subset of $E$.
	
	Observe that $\nu$ is finite since
	\begin{equation*}
		\nu(E) = D^n\delta_\bz(E) + \int_{B^n(\bz,2)}1 + g^p = D^n + \scL^n(B^n(\bz,2)) + \|g\|_p^p  < \infty.
	\end{equation*}
	
	It remains to show that $\nu$ is lower $n$-regular. Let $y = \tilde{f}(x)\in E$ and let $0 < R < \diam(E)$. We want to show that there exists a constant $c_0 > 0$ independent of $x$ and $R$ such that
	\begin{equation*}
		D^n\delta_\bz(B_X(y,R)) + \int_{\tilde{f}^{-1}(B_X(y,R))}(1 + g^p(x))dx \geq c_0R^n.
	\end{equation*}
	If $\scL^n(\td{f}^{-1}(B_X(y,R))) \geq R^n$, then $\nu(B_X(y,R)) \geq \scL^n(\td{f}^{-1}(B_X(y,R))) \geq R^n$. Additionally, if there exists $x\in \tilde{f}^{-1}(B_X(y,R))$ such that $\td{f}(x) = \bz$, then
	\[ \nu(B_X(y,R)) \geq D^n\delta_\bz(B_X(y,R)) = D^n \geq R^n.\]
	Therefore, for the remainder of the proof, we can assume that
	\[ \scL^n(\td{f}^{-1}(B_X(y,R))) < R^n\]
	and that $0\not\in \td{f}^{-1}(B_X(y,R))$. Since $\td{f}|\partial B^n(\bz,2) = \bz$, this latter fact implies that $\td{f}^{-1}(B^n(y,R))$ is a compact set contained in the open ball $U^n(\bz,2)$.
	
	We plan to use Lemma \ref{l:cap-lower-bound}, specifically \eqref{e:cap-leb}, to show that $g^p$ is sufficiently large on $\td{f}^{-1}(B_X(y,R))$. Define $\psi:X\rightarrow \R$ by $\psi(z) = d_X(z,y)$. Since $\td{f}\in W^{1,p}(B^n(\bz,2),X)$ and $\psi$ is 1-Lipschitz, we have $\psi\circ \td{f}\in W^{1,p}(B^n(\bz,2),\R)$ and $|\nabla(\psi\circ \td{f})| \leq g\ \scL^n$-a.e. Define $h:B^n(\bz,2)\rightarrow \R$ by
	\begin{equation*}
		h(u) = \max\left(1 - 2\frac{\psi(\td{f}(u))}{R}, 0\right).
	\end{equation*}
	Define $\Omega_0 = \td{f}^{-1}(U(y,R/2)),\ \Omega_1 = \td{f}^{-1}(U(y,R))$, and let $x\in\Omega_0$ be such that $\td{f}(x) = y$. We claim that $h$ is a continuous function that satisfies the following properties:
	\begin{enumerate}
		\item $h(x) = 1$ and $h(u) = 0$ for all $u\in  \overline{\Omega}_1\setminus\Omega_0$, \label{i:center-outer-values}
		\item $h\in W^{1,p}_0(\Omega_1,\R)$, \label{i:h-compactly-supported}
		\item $(R/2)|\nabla h(u)| \leq g(u)$ for $\scL^n$-a.e. $u\in \overline{\Omega}_1$. \label{i:grad-h-ub}
	\end{enumerate}
	Property (1) follows from the definition of $h$. Property (2) follows from the further claim that $h$ is supported in $\overline{\Omega}_0 \Subset\Omega_1$. For proof of the fact that $\overline{\Omega}_0\Subset\Omega_1$, notice that $\td{f}$ is uniformly continuous on the compact set $\overline{\Omega}_1$ and for any $x_0\in\overline{\Omega}_0$, we have $\dist(\td{f}(x_0),y) \leq R/2$. Choose $\delta > 0$ such that for any $u,v\in\Omega_1$, $|u-v| \leq \delta$ implies $\dist(f(u),f(v)) < R/4$. Now, let $v\in B^n(\overline{\Omega}_0,\delta)$ and let $u\in\overline{\Omega}_0$ be such that $|u-v| \leq \delta$. Then
	\begin{equation*}
		d_X(f(v),y) \leq d_X(f(v),f(u)) + d_X(f(u),y) < \frac{R}{4} + \frac{R}{2} < \frac{3R}{4}.
	\end{equation*}
	This shows that a $\delta$-neighborhood of $\Omega_0$ is contained inside $\Omega_1$, which suffices to prove the claim.
	
	For proof of (3), notice that if $u\in\Omega_0$, then $h(u) = 1 - (2/R)\psi(\td{f}(u))$ so that $R|\nabla h(u)|/2 = |\nabla(\psi\circ \td{f})| \leq g(u)$ a.e. If, instead $u\in\Omega_1\setminus\Omega_0$, then $\psi(\td{f}(u)) > R/2$ so that $h(u) = 0$. This means $\frac{R}{2}|\nabla h(u)| = 0 \leq g(u)$, completing the proof of (3).
	
	Since $h\in W^{1,p}_0(\Omega_1,\R)$, there exists a sequence $h_i\in C^\infty_0(\Omega_1)$ such that $h_i\to h$ in $W^{1,p}(\Omega_1,\R)$. Since $h$ is continuous, any such sequence has $h_i(x)\to 1$. Therefore, there exists $\tilde{h}\in C^\infty_0(\Omega_1)$ such that $\|h-\tilde{h}\|_{W^{1,p}} < \|h\|_{W^{1,p}}/10$ and $\tilde{h}(x) \geq \frac{1}{2}$. Therefore, $2\tilde{h}(x) \geq 1$ so that $2\tilde{h}\in\scR(\{x\},\Omega_1)$, and we can estimate
	\begin{align*}
		\int_{\Omega_1} g^p \geq \frac{R^p}{2^p}\int_{\Omega_1} |\nabla h|^p \geq \frac{R^p}{2^{p+1}}\int_{\Omega_1}|\nabla \tilde{h}|^p &= \frac{R^p}{2^{2p+1}}\int_{\Omega_1}|\nabla (2\tilde{h})|^p\\ &\gtrsim_p R^p\pcap_p(\{x\}, \Omega_1).
	\end{align*}
	Now, \eqref{e:cap-leb} combined with our assumptions that $p > n$ and $\scL^n(\Omega_1) < R^n$ implies
	\begin{equation*}
		\pcap_p(\{x\},\Omega_1) \gtrsim_{n,p}\scL^n(\Omega_1)^{-\frac{p-n}{n}} > R^{n-p}.
	\end{equation*}
	Therefore, we get
	\begin{align*}
		\int_{\Omega_1} g^p \gtrsim_{n,p}R^p\cdot R^{n-p} = R^n.
	\end{align*}
	This proves that $\nu(B_X(y,R))\gtrsim_{n,p}R^n$, so $\nu$ is lower $n$-regular.
\end{proof}

\section{Expansive subsets of metric spaces}

In this section we lay the rest of the groundwork for the proofs of Theorems \ref{thm:main-euc} and \ref{thm:main-met} by introducing \emph{expansive} subsets of metric spaces. In the proofs of both Theorems \ref{thm:main-euc} and \ref{thm:main-met}, we will show that any subset of positive $\Haus^n$ measure on which the conclusion fails is $n$-expansive as given in Definition \ref{def:expansive}. As a consequence of the results of this section, we will be able to conclude that such a set has infinite $n$-packing content, which would give a contradiction to Theorem \ref{t:sob-low-reg}. We will begin with a more general study of $s$-expansive spaces for any $s > 0$ from which we will derive the necessary results.

\subsection{Preliminaries on expansive subsets}
Intuitively, $F\subset X$ is $s$-expansive for $s > 0$ if the following holds: For any $x\in F$ there exists a sequence of scales $r_k\rightarrow0$ and ``$s$-substantial snapshots" $F_{x,r_k}\subset B_X(x,r_k)$ containing $x$ that have some point of $B_X(x,r_k)\setminus F_{x,r_k}$ at distance comparable to $r_k$.

\begin{definition}[$s$-expansive spaces]\label{def:expansive}
	Let $X$ be a metric space, and let $s > 0$. An $\Haus^s$-measurable subset $F\subset X$ is \emph{$s$-expansive in $X$} if $0 < \Haus^s(F) < \infty$ and for $\cH^s$-a.e. $x\in F$, there exists a sequence $r_k\rightarrow0$, a constant $0 < \alpha < 1$, and closed subsets $F_{x,r_k}\subset F\cap B_X(x,r_k)$ such that $x\in F_{x,r_k}$,
	\begin{equation}
		\lim_{k\rightarrow\infty}\frac{\cH^s(F\cap B_X(x,r_k)\setminus F_{x,r_k})}{(2r_k)^s} = 0, \label{item:expansive-dense}
	\end{equation}
	and
	\begin{equation}
		\limsup_{k\rightarrow\infty}\frac{\ex(B_X(x,\alpha r_k), F_{x,r_k})}{r_k} > 0. \label{item:expands}
	\end{equation}
\end{definition}

\begin{example}\label{ex:expansive-1}
\color{white}.\color{black}
	\begin{enumerate}
		\item The set $F = [0,1]^n\times\{0\}$ is $n$-expansive in $X=[0,1]^{n+1}$. In fact, any $F'\subset X$ with $0 < \Haus^n(F') < \infty$ is $n$-expansive in $X$.
		\item If $F\subset [0,1]^n$ with $\Haus^n(F) > 0$, then $F$ is not $n$-expansive in $[0,1]^n$ by the Lebesgue density theorem.\label{item:not-expansive}
		\item In Example \ref{ex:approx-vs-true}, the set $[0,1]^n\times\{0\}$ is $n$-expansive in $X$.
		\item In Example \ref{ex:rect-no-tan-planes}, the set $D = [0,1]^n\times\{0\}^{d-n}$ is $n$-expansive in $E$.
	\end{enumerate}
\end{example}

As we will see, the motivation for the name is that an $s$-expansive subset $F$ ``expands'' $X$ nearby $F$, augmenting merely the existence of $F\subset X$ with $0 < \Haus^s(F) <\infty$ into $P_\infty^s(X) = \infty$. This condition can also be viewed as a type of local porosity condition for $F$ inside $X$.

\begin{example}\label{ex:expansive-2}
	Let $X$ be a separable metric space with $0 < \Haus^s(X) < \infty$ for some $s > 0$. Let $\{x_n\}_{n\in\N}$ be a countable dense subset of $X$, and let $(y_n)_{n\in\N}\subset[0,1]$ be any sequence with $y_n\rightarrow 0$. Consider the metric space
	\begin{equation*}
		Y = (X \times\{0\}) \cup (\{x_n:n\in\N\} \times \{y_n:n\in\N\}) \subset X\times[0,1]
	\end{equation*}
	with the product metric. The set $F = X\times\{0\}$ is $s$-expansive in $X$.
\end{example}

To establish the main lemma on $s$-expansivity, we will need to use Vitali coverings and the notion of an aymptotically doubling measure.
\begin{definition}[Vitali cover]
	A family $\mathscr{B}$ of closed balls on a metric space $X$ is a \emph{Vitali cover} of $E\subset X$ if for every $x\in E$ and $\epsilon > 0$, there exists $0 < r < \epsilon$ such that $B_X(x,r)\in \scB$.
\end{definition}
\begin{definition}[Asymptotically doubling measure]
	A Borel regular measure $\mu$ on a metric space $X$ is \emph{asymptotically doubling} if
	\begin{equation*}
		\limsup_{r\rightarrow 0}\frac{\mu(B_X(x,2r))}{\mu(B_X(x,r))} < \infty \hbox{\quad for $\mu$-a.e. $x\in X$.}
	\end{equation*}
\end{definition}
The following lemma tells us that we have access to the Vitali covering theorem as long as our measure is asymptotically doubling.
\begin{lemma}[{\cite[Theorem 2.10]{Ba22}}]\label{l:rect-vitali}
	If $\Haus^s$ is asymptotically doubling at $\Haus^s$-a.e. $x\in X$, then $\Haus^s$ has a Vitali covering theorem. That is, if $E\subset X$ has a Vitali cover $\scB$, then there exists $\scB'\subset\scB$ such that $\scB'$ is disjoint, countable, and  $\Haus^s\left( E\setminus\bigcup\scB' \right) = 0$.
\end{lemma}
We will use the following lemma to produce useful Vitali covers for $s$-expansive spaces with positive lower $s$-density. We will need to restrict to such spaces for the rest of the section. Given a metric space $X$, a Borel set $A\subset X$ and $x\in\X$, recall the upper and lower, respectively, Hausdorff densities
\[ \Theta^{*,s}(A,x) = \limsup_{r\to 0} \frac{\mathcal{H}^s(B(x,r))}{(2r)^s}, \quad \Theta^{s}_*(A,x) = \liminf_{r\to 0}\frac{\mathcal{H}^s(B(x,r))}{(2r)^s}.\]

\begin{lemma}\label{l:ex-helper}
	Let $X$ be a metric space, and let $\td{F}\subset X$ be $s$-expansive in $X$. Fix choices of closed sets $\{\tilde{F}_{x,r_k}\}_{x\in F,k\geq0}$ satisfying \eqref{item:expansive-dense} and \eqref{item:expands}, and define for each $\epsilon > 0$
	\begin{align*}
		F^\epsilon = \bigg\{x\in \td{F} : \limsup_{k\to \infty}\frac{\ex(B_X(x,(1-\epsilon)r_k),\td{F}_{x,r_k})}{r_k} > \epsilon, \ &\Theta^s_*(\td{F},x) > \epsilon,\\
		&\text{and }\Theta^{*s}(\td{F},x) \leq 1\bigg\}.
	\end{align*}
	Let $F\subset F^\epsilon$, let $\delta > 0$, and let $\scB(\epsilon,\delta)$ be the collection of balls in $X$ such that for any $B_X(x,r)\in\scB(\epsilon,\delta)$, there exists a closed set $F_{x,r}\subset F\cap B_X(x,r)$ satisfying
	\begin{enumerate}[label=(\roman*)]
		\item \label{i:GB-euc-big} $(1-\delta)\Haus^s(F\cap B_X(x,r)) \leq \Haus^s(F_{x,r}) \leq (1+\delta)(2r)^s$,
		\item \label{i:euc-far-point} $B_X(y,\epsilon r)\subseteq B_X(x,r)$ and $\dist(y,F_{x,r}) > \epsilon r$, for some $y\in B_X(x,r)\setminus F_{x,r}$.
	\end{enumerate}
	Then the collection $\scB(\epsilon,\delta)$ is a Vitali cover of a full $\Haus^s$-measure subset of $F$.
\end{lemma}

\begin{proof}
	Let $\delta > 0$ and $F\subset F^{\epsilon}$. For every $x\in F$ and $k\geq0$, set $F_{x,r_k} = \td{F}_{x,r_k}\cap F$.
	At $\Haus^s$-a.e. $x\in F$, we get
	\begin{align*}
		\lim_{k\rightarrow\infty}\frac{\cH^s(F\cap B_X(x,r_k)\setminus F_{x,r_k})}{(2r_k)^s} &= \lim_{k\rightarrow\infty}\frac{\cH^s(F\cap B_X(x,r_k)\setminus \td{F}_{x,r_k})}{(2r_k)^s}\\
		&\leq \lim_{k\rightarrow\infty}\frac{\cH^s(\td{F}\cap B_X(x,r_k)\setminus \td{F}_{x,r_k})}{(2r_k)^s} = 0,
	\end{align*}
	and
	\begin{equation*}
		\Theta^{s}_*(F, x) = \Theta^s_*(\td{F},x) > \epsilon \text{\quad and \quad} \Theta^{*s}(F,x) = \Theta^{*s}(\td{F},x) \leq 1.
	\end{equation*}
	This means the following hold at $\Haus^s$-a.e. $x\in F$ for every sufficiently large $k$:
	\begin{enumerate}
		\item \label{i:Bx-msr} $\epsilon (2r_k)^s \leq \Haus^s(F\cap B_X(x,r_k)) \leq (1+\delta)(2r_k)^s$, and
		\item \label{i:E-msr} $\Haus^s(F\cap B_X(x,r_k)\setminus F_{x,r_k}) \leq \delta\epsilon(2r_k)^s \leq \delta\Haus^s(F\cap B_X(x,r_k))$.
	\end{enumerate}
	We trivially get $\Haus^s(F_{x,r_k}\cap B_X(x,r_k)) \leq (1+\delta)(2r_k)^s$, from the right inequality in \eqref{i:Bx-msr}, while \eqref{i:E-msr} gives
	\begin{align*}
		\Haus^s(F_{x,r_k}) &= \Haus^s(F\cap B_X(x,r)) - \Haus^s(F\cap B_X(x,r)\setminus F_{x,r_k})\\
		& \geq (1-\delta)\Haus^s(F \cap B_X(x,r_k)).
	\end{align*}
	We can additionally find infinitely many $k$ such that
	\begin{equation*}
		\ex(B_X(x,(1-\epsilon)r_k),F_{x,r_k}) \geq \ex(B_X(x,(1-\epsilon)r_k),\td{F}_{x,r_k}) > \epsilon r_k,
	\end{equation*}
	which implies that there exists $y\in B_X(x,(1-\epsilon)r_k)$ such that $\dist(y,F_{x,r_k}) > \epsilon r_k$. Therefore, $B_X(x,r_k)\in\scB(\epsilon,\delta)$. This shows that $\scB(\epsilon,\delta)$ is a Vitali cover of $F$.
\end{proof}

\subsection{Expansive subsets give infinite packing content}
In the following proposition, we iterate Lemma \ref{l:ex-helper} to show that inside any metric space with an $s$-expansive subset with positive lower $s$-density, we can pack a large tree of disjoint balls into $X$. Each level of balls in the tree will contribute a fixed amount to the $s$-packing content, showing that $X$ has infinite $s$-packing content.
\begin{proposition}\label{p:expansive}
	Let $X$ be a metric space. If there exists $F\subset X$ that is $s$-expansive in $X$ and $\Theta^s_*(F,x) > 0$ for $\Haus^s$-a.e. $x\in F$, then $X$ does not have finite $s$-packing content.
\end{proposition}

\begin{proof}
	For any $x\in F$, fix a sequence $r_k\rightarrow0$ and a choice of subsets $F_{x,r_k}\subseteq B_X(x,r_k)$ satisfying \eqref{item:expansive-dense} and \eqref{item:expands}. The inequality \eqref{item:expands} and the lower $s$-density assumption implies that there exists $0 < \epsilon < \frac{1}{2}$ such that the set
	\begin{align*}F^\epsilon = \bigg\{x\in F : \limsup_{k}\frac{\ex(B_X(x,(1-\epsilon)r_k),F_{x,r_k})}{r_k} > \epsilon,\ &\Theta^s_*(F,x) > \epsilon,\\
		&\text{and }\Theta^{*s}(F,x) \leq 1\bigg\}
	\end{align*}
	has $\Haus^s(F^\epsilon) > 0$. Let $\delta_k = 2^{-k}\frac{1}{100}$ for all integers $k \geq 0$.
	
	We wish to apply Lemma \ref{l:ex-helper} to $F^\epsilon$ itself. Let $\scB_0' := \scB(\epsilon,\delta_0)$ be the family of balls centered on $F^\epsilon$ that satisfy all the conditions in Lemma \ref{l:ex-helper} for $\delta=\delta_0$. The collection $\scB_0'$ is a Vitali cover for $F^\epsilon$, and the fact that
	\begin{equation*}
		0 < \Theta^s_*(F,x) = \Theta^s_*(F^\epsilon,x) \leq \Theta^{*s}(F^\epsilon,x) = \Theta^{*s}(F,x) < \infty
	\end{equation*}
	at $\Haus^s$-a.e. $x\in F^\epsilon$ implies $\Haus^s$ is asymptotically doubling on $F^\epsilon$. Lemma \ref{l:rect-vitali} gives a finite disjoint collection $\scB_0\subset\scB_0'$ such that
	\begin{equation}\label{e:good-msr}
		\Haus^s\left(F^\epsilon \setminus \textstyle\bigcup\scB_0\right) \leq \delta_0\Haus^s(F^\epsilon),
	\end{equation}
	and for any $B\in\scB_0$ there exists a closed subset $F_B\subset F^\epsilon\cap B$ satisfying
	\begin{enumerate}
		\item $(1-\delta)\Haus^s(F^\epsilon\cap B) \leq \Haus^s(F_B) \leq (1+\delta)(2r(B))^s$,
		\item There exists $y_B\in B_X(x,r)\setminus F_B$ such that $B_X(y_B,\epsilon r)\subseteq B_X(x,r)$ and $\dist(y,F_B) \geq \epsilon r$.
	\end{enumerate}
	This means that
	\begin{equation}\label{e:B0-bound}
		\sum_{B\in\scB_0}(2r(B))^s \geq \sum_{B\in\scB_0}\frac{1-\delta_0}{1+\delta_0}\Haus^s(F^\epsilon\cap B) \geq \frac{1}{2}\Haus^s(F^\epsilon),
	\end{equation}
	where we used \eqref{e:good-msr} in the second inequality. Because $\scB_0$ is a finite disjoint collection of closed balls, we can also find $\eta_0 > 0$ such that $(1+10\eta_0)B \cap (1+10\eta_0)B' = \varnothing$ for any $B,B'\in\scB_0$.
	
	Assuming that we have defined finite disjoint families $\scB_{j}$ and decreasing constants $\eta_j$ for all $0 \leq j \leq k-1$ as above, we now define a new family $\scB_{k}$ by iterating this construction inside each ball of $\scB_{k-1}$ to form an infinite tree of balls. For any ball $B = B_X(x,r)\in\scB_{k-1}$ with subset $F_B\subset F^\epsilon\cap B_X(x,r)$, we let $\scB'_{k}(B)$ be the collection of balls $B' = B_X(y,t)$ centered on $F_B$ such that $t <\min\{\epsilon r, \eta_{k-1}r\}$, and there exists a closed subset $F_B'\subset F_B$ with
	\begin{enumerate}
		\item \label{i:GB-exists} $(1-\delta_k)\Haus^s(F_{B}\cap B') \leq \Haus^s(F_{B'}) \leq (1+\delta_k)(2t)^s$,
		\item There exists $y_{B'}\in B'\setminus F_{B'}$ such that $B_X(y_{B'},\epsilon t)\subseteq B'$ and $\dist(y_{B'},F_{B'}) \geq \epsilon t$.
	\end{enumerate}
	Lemma \ref{l:ex-helper} implies that $\scB_{k}'(B)$ is a Vitali cover of a full measure subset of $F_B$. We get a finite, disjoint subfamily $\scB_{k}(B) \subset \scB_{k}'(B)$ such that
	\begin{equation}\label{e:GB-good-msr}
		\Haus^s\left(F_B \setminus\textstyle\bigcup\scB_{k}(B)\right) \leq \delta_{k}\Haus^s(F_B)
	\end{equation}
	and a constant $\eta_k > 0$ such that $(1+10\eta_k)B' \cap (1+10\eta_k)B'' = \varnothing$ for any $B',B''\in \scB_k(B)$.
	
	Now, define
	\begin{equation*}
		\scB_{k} = \bigcup_{B\in\scB_{k-1}}\scB_{k}(B).
	\end{equation*}
	This completes the definition of $\scB_k$ for all $k\geq0$. We now define
	\begin{equation}
		\scB = \bigcup_{k=0}^\infty\scB_k.
	\end{equation}
	We will finish the proof by showing that the existence of the collection $\scB$ implies $X$ has infinite $s$-packing content. For each $k\in\N$ and $B\in \scB_k$, there exists $y_B\in B$ such that $B_X(y_B,\epsilon r(B))\subset B$ and $\dist(y_B,F_B) \geq \epsilon r(B)$. This condition guarantees that the ball $T_B = B_X(y_B,\epsilon r(B))$ satisfies $T_B\cap B' = \varnothing$ for any $B'\in\scB_{m}$ with $m > k$. On the other hand, the choice of $\eta_k$ ensures that any child ball $B'$ of $B\in\scB_m$ remains inside $(1+10\eta_m)B$, ensuring that $T_{B'}\cap B'' = \varnothing$ for any $B''\in\scB_m$ with $m \leq k$. This means that $T_B \cap T_{B'} = \varnothing$ for any balls $B,B'\in \cup_k\scB_k'$ with $B\not= B'$. We now want to show that the sum of the radii of the disjoint balls in $\{T_B\}_{B\in\scB}$ is infinite by lower-bounding the sum of the radii of balls in $\scB$.
	
	For each $k\in\N$, applying \eqref{i:GB-exists} and then \eqref{e:GB-good-msr} gives
	\begin{align*}
		\sum_{B\in\scB_k} \Haus^s(F_B) &= \sum_{B'\in\scB_{k-1}}\sum_{B\in\scB_k(B')}\Haus^s(F_B) \\
		&\geq  \sum_{B'\in\scB_{k-1}}\sum_{B\in\scB_k(B')}(1-\delta_k)\Haus^s(F_{B'}\cap B)\\
		&\geq (1-\delta_{k-1})^2\sum_{B'\in\scB_{k-1}}\Haus^s(F_{B'}).
	\end{align*}
	Using \eqref{i:GB-exists}, then iterating the previous inequality and using \eqref{e:B0-bound} gives
	\begin{align*}
		\sum_{B\in\scB_{k}(B)}(2r(B))^s \geq (1-\delta_k)^2\sum_{B\in\scB_k} \Haus^s(F_B) &\geq \bigg(\prod_{j=0}^{k}(1-\delta_j)^2\bigg)\sum_{B\in\scB_0}\Haus^s(F_B)\\
		&\gtrsim \Haus^s(F^\epsilon).
	\end{align*}
	We can now conclude that $X$ has infinite $s$-packing content.
	Indeed,
	\[
	\sum_{k=0}^\infty \sum_{B\in\scB_k}(2r(T_B))^s \gtrsim_{s,\epsilon} \sum_{k=0}^\infty\sum_{B\in\scB_k}r(B)^s\gtrsim_s \sum_{k=0}^\infty \Haus^s(F^\epsilon) = \infty.  \qedhere
	\]
\end{proof}

\section{Tangents and approximate tangents}

In this section we explore \emph{approximate tangents}, a notion weaker than that of tangents, where instead of the entire set $E\subset \R^d$ (for the Euclidean case) or the space $X$ (for the metric case), only ``substantial'' subsets are taken into account. Our definitions of approximate tangents provide generalizations of the well-known approximate tangent planes used in the study of rectifiable sets.

\subsection{Approximate Attouch-Wets tangents}
Intuitively, a set $T\subset\R^d$ is an approximate $s$-tangent of $E\subset\R^d$ at $x\in E$ if there exists a sequence of ``$s$-substantial snapshots'' $E_{x,r_k}\subset E$ containing $x$ such that $E_{x,r_k}$ approaches $T$ in both measure and Walkup-Wets distance.

\begin{definition}[approximate Attouch-Wets tangents]\label{def:approx-tan-plane}
	Let $E\subset\R^d$ and $s>0$. A set $T\in\mathfrak{C}(\R^d,\bz)$ is an \emph{approximate $s$-tangent} at $x\in E$ if
	there exists a sequence $r_i\rightarrow0$ and subsets $E_{x,r_i}\subset E$ with $x\in E_{x,r_i}$ such that
	\begin{equation}\label{e:euc-apptan-density}
		\lim_{i\rightarrow\infty}\frac{\Haus^s(E\cap B^d(x,Rr_i)\setminus E_{x,r_i})}{(2Rr_i)^{s}}= 0 \hbox{\quad for every $R > 0$}
	\end{equation}
	and
	\begin{equation}\label{e:euc-apptan-distance}
		\lim_{i\rightarrow\infty} \D^{\bz,R}[r_i^{-1}(E_{x,r_i}-x),T] = 0 \hbox{\quad for every $R > 0$.}
	\end{equation}
	If the above holds for \emph{all} sequences $r_i\rightarrow 0$, then we say that $T$ is a \emph{strong approximate $s$-tangent} at $x\in E$. We denote the set of all approximate $s$-tangents by $\appTanAW(E,x,s)$.
\end{definition}

\begin{proposition}[upgrading Attouch-Wets approximate $s$-tangents]\label{p:euc-apptan-upgrade}
	Let $E\subseteq\R^d$, and let $s > 0$. If $E$ has finite $s$-packing content, then the following hold at $\Haus^s$-a.e. $x\in E$:
	\begin{enumerate}
		\item \label{item:euc-apptan-equals-tan}
		$\appTanAW(E,x,s) = \TanAW(E,x)$,
		\item \label{item:euc-strong-apptan} if $T_x$ is a strong approximate $s$-tangent at $x$, then $\TanAW(E,x) = \{T_x\}.$
	\end{enumerate}
\end{proposition}

\begin{proof}
	We begin with \eqref{item:euc-apptan-equals-tan}. Towards a contradiction, suppose there exists an $\Haus^s$-measurable set $F\subset E$ with $0 < \Haus^s(F) < \infty$ on which \eqref{item:euc-apptan-equals-tan} fails. We claim that $F$ is $s$-expansive in $E$. This will finish the proof since then Proposition \ref{p:expansive} will imply that $E$ has infinite $s$-packing content, contradicting our hypothesis. Indeed, we know from Corollary 2.4 in \cite{BLZ23} that finite $s$-packing content implies that $\Theta^{s}_*(F,x) > 0$ for $\Haus^s$-a.e. $x\in F$, meaning that $F$ would satisfy the hypotheses of the proposition.
	
	Let $x\in F$, and let $T\in\appTanAW(E,x,s)\setminus \TanAW(E,x)$. By definition of approximate tangents, there exists a sequence $r_i\rightarrow 0$ and subsets $E_{x,r_i}\subset E$ such that \[\lim_{i\rightarrow\infty}\frac{\Haus^s(E\cap B^d(x,Rr_i)\setminus E_{x,r_i})}{(2Rr_i)^s} = 0 \hbox{\quad for every $R > 0$,}\]
	and
	\[\lim_{i\rightarrow\infty} \sD^{\bz,R}[r_i^{-1}(E_{x,r_i}-x),T] = 0 \hbox{\quad for every $R > 0$.}\]
	Since $T\not\in\TanAW(E,x)$, Lemma \ref{lem:AW-char} implies there exist some $R_0 > 0$ and $\epsilon > 0$ such that
	\[\limsup_{i\rightarrow\infty} \sD^{\bz,R_0/4}[r_i^{-1}(E-x),T] > \epsilon.\]
	Set $s_i = R_0r_i$ and $E_{x,s_i} =  E_{x,r_i}\cap B(x,s_i)$. We claim that $E_{x,s_i}$ satisfy \eqref{item:expansive-dense} and \eqref{item:expands}. Since \eqref{e:euc-apptan-density} implies \eqref{item:expansive-dense} immediately, we only need to prove that \eqref{item:expands} holds.
	
	The previous two displayed statements imply that there exist infinitely many $i$ such that
	\begin{equation*}
		D^{\bz,R_0/4}[r_i^{-1}(E-x),T] > \epsilon \text{\quad and \quad} D^{\bz,R_0}[ r_i^{-1}(E_{x,r_i}-x),T] < \frac{\epsilon}{8}.
	\end{equation*}
	Applying Lemma \ref{l:mt-ex-lb}, we see that for each such $i$,
	\begin{align*}
		\ex(E\cap B^d(x,s_i/2), E_{x,s_i}) &= \ex(E\cap B^d(x,R_0r_i/2), \ E_{x,r_i})\\
		&=r_i\ex(r_i^{-1}(E-x) \cap B^d(\bz,R_0/2), r_i^{-1}(E_{x,r_i}-x))\\
		&> r_i \bigg(\frac{\epsilon R_0}{8}\bigg)= \frac{\epsilon}{8}s_i.
	\end{align*}
	To prove \eqref{item:euc-strong-apptan}, we proceed similarly by contradiction. We receive a set $F\subset E$ such that $0 < \Haus^s(F) < \infty$ on which \eqref{item:euc-strong-apptan} fails. It suffices to prove that $F$ is $s$-expansive in $E$ since applying Proposition \ref{p:expansive} would then give a contradiction.
	
	Let $x\in F$, and let the subsets $E_{x,r}\subset E$ for $r > 0$ produce the strong approximate tangent $T_x$ at $x$. By \eqref{item:euc-apptan-equals-tan}, we can also assume that $T_x\in \Tan_\AW(E,x)$ for every $x\in F$. The strong version of \eqref{e:euc-apptan-distance} says
	\[\lim_{r\rightarrow0} \D^{\bz,R}[r^{-1}(E_{x,r}-x),T_x] = 0 \hbox{\quad for every $R > 0$.}\]
	Since $T_x\in \TanAW(E,x)\not= \{T_x\}$, there exists $T_x'\not=T_x$ with $T_x'\in\Tan_\AW(E,x)$ and a sequence $r_i\rightarrow0$ such that
	\[\lim_{i\rightarrow\infty}\D^{\bz,R}[r_i^{-1}(E-x),T_x'] = 0.\]
	This means that there exist $R_0 > 0$ and $\epsilon > 0$, such that
	\[\limsup_{i\rightarrow\infty}\D^{\bz,R_0/4}[r_i^{-1}(E-x),T_x] > \epsilon.\]
	Indeed, otherwise Lemma \ref{lem:AW-char} would imply $r_i^{-1}(E-x)$ converges to both $T_x$ and $T_x'$ in the Attouch-Wets topology. Set $s_i = R_0r_i$ and let $E_{x,s_i} = E_{x,r_i}\cap B^d(x,s_i)$. We claim that the subsets $E_{x,s_i}$ satisfy \eqref{item:expansive-dense} and \eqref{item:expands}. Since \eqref{item:expansive-dense} follows immediately from \eqref{e:euc-apptan-density}, we only need to prove \eqref{item:expands} holds. We find infinitely many $i > 0$ such that
	\[D^{\bz,R_0/4}[r_i^{-1}(E-x),T_x] > \epsilon \text{\quad and \quad} D^{\bz,R_0}[ r_i^{-1}(E_{x,r_i}-x),T_x] < \frac{\epsilon}{8}.\]
	Applying Lemma \ref{l:mt-ex-lb} as in the previous argument, we see that
	\[ \limsup_{i\rightarrow\infty} s_i^{-1}\ex(E\cap B^d(x,s_i/2), E_{x,s_i}) > \frac{\epsilon}{8}. \qedhere\]
\end{proof}

\subsection{Approximate metric tangents} Here we define the metric analogue of approximate Attouch-Wets tangents.

\begin{definition}[approximate metric tangents]
	Let $X$ be a metric space. A pointed metric space $(Y,y)$ is an \emph{approximate $s$-tangent} at $x\in X$ if there exists a sequence $r_i\rightarrow0$ such that for every $i\geq 0$,
	there exists $X_{x,r_i}\subset X$ with $x\in X_{x,r_i}$ satisfying
	\begin{equation}\label{e:mt-apptan-density}
		\lim_{i\rightarrow\infty}\frac{\Haus^s(B_X(x,Rr_i)\setminus X_{x,r_i})}{(2Rr_i)^{s}}= 0 \hbox{\quad for every $R > 0$},
	\end{equation}
	and
	\begin{equation}\label{e:mt-apptan-distance}
		\lim_{i\rightarrow\infty} \D_\GH^{R}[(r_i^{-1}X_{x,r_i},x),(Y,y)] = 0 \hbox{\quad for every $R > 0$.}
	\end{equation}
	If the above holds for \emph{every} sequence $r_i\rightarrow 0$, then we say that $(Y,y)$ is a \emph{strong approximate $s$-tangent} at $x\in X$. We denote the set of pointed isometry classes of all approximate $s$-tangents at $x$ by $\appTanGH(X,x,s)$.
\end{definition}

When we try to upgrade approximate metric tangents, we run into a complication not found in the Euclidean case. We want to conclude that small distance of a substantial set $X_{x,r}$ to a tangent $Y$ and large distance of $B_X(x,r)$ to $Y$ imply an excess lower bound of $X$ over $X_{x,r}$. We need to use particular isometric embeddings of $X_{x,r}$ and $Y$ in some common space $Z$ for which the distance is small, and use this same embedding structure to get a Walkup-Wets lower bound for $B_X(x,r)$ and $Y$. The following lemma allows us to extend the isometric embedding from $X_{x,r}$ to $X$ into an extended common space $\hat{Z}$, allowing us to measure mutual distances of $X,X_{x,r},$ and $Y$ appropriately.

\begin{lemma}\label{l:extension-for-dgh}
	Let $(X,x), (Y,y), (Z,z)$ be pointed metric spaces, and let $\td{X}\subset X$ with $x\in\td{X}$. Let $\iota_{\td{X}}:(\td{X},x)\rightarrow(Z,z)$ and $\iota_Y:(Y,y)\rightarrow(Z,z)$, be pointed isometric embeddings, and identify $\td{X}$ and $Y$ with their images inside $Z$. The set
	\begin{equation}
		\hat{Z} = Z\sqcup(X\setminus\td{X})
	\end{equation}
	with the metric
	\begin{equation}
		d_{\hat{Z}}(u,v) = \begin{cases}
			d_X(u,v) & \hbox{\quad if $u,v\in X$},\\
			d_Z(u,v) & \hbox{\quad if $u,v\in Z$},\\
			\inf_{\td{x}\in\td{X}} d_X(u,\td{x}) + d_Z(\td{x},v) & \hbox{\quad if $u\in X,\ v\in Z$.}
		\end{cases}
	\end{equation}
	is a metric space. The pointed space $(\hat{Z},z)$ contains isometric copies of $X$, $\td{X}$, and $Y$, and $d_{\hat{Z}}(u,v) = d_Z(u,v)$ for all $u\in\td{X}$ and $v\in Y$. In particular, there exists an isometric embedding $\iota_X:X\rightarrow\hat{Z}$ such that $\iota_X(x')= \iota_{\td{X}}(x')$ for all $x'\in\td{X}$.
\end{lemma}

\newcommand{\hrho}{\hat{\rho}}

\begin{proof}
	The only nontrivial claim is that $d_{\hat{Z}}$ is a metric, and the only nontrivial property to prove is that $d_{\hat{Z}}$ satisfies the triangle inequality. Since $d_{\hat{Z}}$ is a metric on $Z$ and $X$ separately, we only need to consider triples $u,v,w\in\hat{Z}$ for which either two points are in $X$ and one point is in $Z$ or for which two points are in $Z$ and one is in $X$.
	
	First, let $u,v\in X$ and let $w\in Z$. We have
	\begin{align*}
		d_{\hat{Z}}(u,w) + d_{\hat{Z}}(w,v) &= \inf_{\td{x}\in\td{X}}d_X(u,\td{x}) + d_Z(\td{x},w) + \inf_{\td{x}'\in\td{X}}d_X(v,\td{x}') +d_Z(\td{x}',w)\\
		&= \inf_{\td{x},\td{x}'\in\td{X}} d_X(u,\td{x}) + d_X(v,\td{x}') + (d_Z(\td{x},w) + d_Z(\td{x}',w))\\
		&\geq \inf_{\td{x},\td{x}'\in\td{X}}d_X(u,\td{x}) + d_X(v,\td{x}') + d_Z(\td{x},\td{x}')\\
		& =  \inf_{\td{x},\td{x}'\in\td{X}}d_X(u,\td{x}) + d_X(\td{x},\td{x}')+ d_X(\td{x}',v)\\
		&\geq d_X(u,v)= d_{\hat{Z}}(u,v).
	\end{align*}
	Now, suppose $u\in X$ and $v,w\in Z$. Then
	\begin{align*}
		d_{\hat{Z}}(u,w) + d_{\hat{Z}}(w,v) &= \inf_{\td{x}\in\td{X}}d_X(u,\td{x}) + d_Z(\td{x},w) + d_Z(w,v)\\
		&\geq \inf_{\td{x}\in\td{X}}d_X(u,\td{x}) + d_Z(\td{x},v)\\
		&= d_{\hat{Z}}(u,v).
	\end{align*}
	The case when $u,v\in Z$ and $w\in X$ is similar to the first case above, and the case when $u\in Z$ and $v,w\in X$ is similar to the second case above.
\end{proof}

\begin{proposition}[upgrading approximate metric tangents]\label{p:mt-apptan-upgrade}
	Let $X$ be a metric space, and let $s > 0$. If $X$ has finite $s$-packing content, then the following hold at $\Haus^s$-a.e. $x\in X$:
	\begin{enumerate}
		\item \label{item:mt-apptan-equals-tan}
		$\appTanGH(X,x,s) = \TanGH(X,x)$,
		\item \label{item:mt-unique-strong-apptan} if $(Y,y)$ is a strong approximate $s$-tangent at $x$, then 
		\begin{equation}
			\TanGH(X,x) = \{[(Y,y)]\}.
		\end{equation}
	\end{enumerate}
\end{proposition}

\begin{proof}This is similar to the proof of Proposition \ref{p:euc-apptan-upgrade}, but there are some complications coming from the fact that convergence is measured using $\D_\GH$ rather than $\D$ as discussed prior to Lemma \ref{l:extension-for-dgh}. We begin with \eqref{item:mt-apptan-equals-tan}. Towards a contradiction, suppose there exists an $\Haus^s$-measurable set $F\subset X$ with $0 < \Haus^s(F) < \infty$ on which \eqref{item:mt-apptan-equals-tan} fails. We claim that $F$ is $s$-expansive in $X$. This will finish the proof since then Proposition \ref{p:expansive} will imply that $X$ has infinite $s$-packing content, giving a contradiction to our hypothesis. Indeed, we know from Corollary 2.4 in \cite{BLZ23} that finite $s$-packing content implies that $\Theta^{s}_*(F,x) > 0$ for $\Haus^s$-a.e. $x\in F$, meaning that $F$ would satisfy the hypotheses of the proposition.
	
	Let $x\in F$, and let $[(Y,y)]\in\appTanGH(X,x,s)\setminus \TanGH(X,x)$. Let $r_i\rightarrow0$ and  let $X_{x,r_i}\subset X$ produce the approximate tangent $(Y,y)$.
	Since $[(Y,y)]\not\in\TanGH(X,x)$, Lemma \ref{l:dgh-basics} implies that there exists $R_0 > 0$ and $\epsilon > 0$ such that
	\begin{equation}\label{eq:dgh-bdd-below}
		\limsup_{i\rightarrow\infty} \D_\GH^{R_0/4}[(r_i^{-1}X,x),(Y,y)] > \epsilon.
	\end{equation}
	Let $s_i = R_0r_i$ and define $X_{x,s_i} = X_{x,r_i} \cap B_X(x,s_i)$. We claim that $X_{x,s_i}$ satisfy \eqref{item:expansive-dense} and \eqref{item:expands}. Since \eqref{item:expansive-dense} follows immediately from \eqref{e:mt-apptan-density}, we only need to prove \eqref{item:expands} holds.
	
	Using the definition of $\D_\GH$ and the fact that
	\[\lim_{i\rightarrow\infty}\D_\GH^{R}[(r_i^{-1}X_{x,r_i},x),(Y,y)] = 0 \hbox{\quad for every $R > 0$},\]
	for any sufficiently large $i$, we can find a pointed metric space $(Z,z)$ and pointed isometries $\iota_1:(r_i^{-1}X_{x,r_i},x)\rightarrow(Z,z)$ and $\iota_2:(Y,y)\rightarrow(Z,z)$ for which
	\begin{equation}\label{eq:dgh-bdd-above}
		\D^{z,R_0}[\iota_1(r_i^{-1}X_{x,r_i}), \iota_2(Y)] < \frac{\epsilon}{8}.
	\end{equation}
	Applying Lemma \ref{l:extension-for-dgh}, we get a metric space $\hat{Z}$ extending $Z$ and an isometric embedding $\iota_X:(r_i^{-1}X,x)\rightarrow(Z,z)$ such that $\iota_X(x') = \iota_1(x')$ for all $x'\in r_i^{-1}X_{x,r_i}$. Identifying $r_i^{-1}X, r_i^{-1}X_{x,r_i},$ and $Y$ with their isometric copies inside $\hat{Z}$ and using \eqref{eq:dgh-bdd-below} and \eqref{eq:dgh-bdd-above}, we find infinitely many $i$ such that
	\begin{equation}\label{eq:mt-ww-bounds}
		\D^{z,R_0/4}[r_i^{-1}X, Y] > \epsilon \text{\quad and \quad} \D^{z,R_0}[ r_i^{-1}X_{x,r_i}, Y] < \frac{\epsilon}{8}.
	\end{equation}
	Lemma \ref{l:mt-ex-lb} then implies
	\begin{align*}\label{eq:mt-ex-lb}
		\ex(B_X(x,s_i/2), X_{x,s_i}) &= \ex(B_X(x,R_0r_i/2), X_{x,r_i})\\
		&=  r_i\ex(r_i^{-1}X \cap B_{\hat{Z}}(z,R_0/2),  r_i^{-1}X_{x,r_i})\\\nonumber
		&\geq r_i \bigg(\frac{R_0\epsilon}{8}\bigg)= \frac{\epsilon}{8}s_i.
	\end{align*}
	This finishes the proof of \eqref{item:expands}.
	
	To prove \eqref{item:mt-unique-strong-apptan}, we proceed similarly by contradiction. We receive a set $F\subset X$ with $0 < \Haus^s(F) < \infty$ on which \eqref{item:mt-unique-strong-apptan} fails. It suffices to prove that $F$ is $s$-expansive in $X$ since applying Proposition \ref{p:expansive} would then give a contradiction.
	
	 Let $x\in F$ and let subsets $X_{x,r}\subset X$ for $r > 0$ produce the strong approximate tangent $(Y_x,y)$ at $x\in F$. By \eqref{item:mt-apptan-equals-tan}, we can also assume that $[(Y_x,y)]\in\Tan_{\GH}(X,x)$ for every $x\in F$. By definition,
	\[\lim_{r\rightarrow0} \D_\GH^{R}[(r^{-1}X_{x,r},x),(Y_x,y)]  = 0 \hbox{\quad for every $R > 0$}.\]
	Since $\{[(Y_x,y)]\}\subsetneq\Tan_{\GH}(X,x)$, there exists a pointed isometry class $[(Y_x',y')]\not=[(Y_x,y)]$ such that $[(Y_x',y')]\in \Tan_\GH(X,x)$. Using Lemma \ref{l:dgh-basics}, there exists $r_i\rightarrow0$ such that
	\[\lim_{i\rightarrow\infty}\D_\GH^{R/4}[(r_i^{-1}X,x), (Y_x',y')] = 0 \hbox{\quad for every $R > 0$}.\]
	This means there exist $R_0 > 0$ and $\epsilon > 0$ such that
	\[\limsup_{i\rightarrow\infty}\D_{\GH}^{R_0/4}[(r_i^{-1}X,x),(Y_x,y)] > \epsilon.\]
	Indeed, otherwise $(r_i^{-1}X,x)$ would converge to both $(Y_x,y)$ and $(Y_x',y')$ in the pointed Gromov-Hausdorff topology. Set $s_i = R_0r_i$ and let $X_{x,s_i} = X_{x,r_i}\cap B(x,s_i)$. We claim that $X_{x,s_i}$ satisfy \eqref{item:expansive-dense} and \eqref{item:expands}. Since \eqref{item:expansive-dense} follows immediately from \eqref{e:mt-apptan-density}, we only need to prove \eqref{item:expands} holds.
	
	As in the proof of \eqref{item:mt-apptan-equals-tan}, for any $i > 0$ we use Lemma \ref{l:extension-for-dgh} to produce a common metric space $(\hat{Z},z)$ containing isometric copies of $r_i^{-1}X, r_i^{-1}X_{x,r_i},$ and $Y_x$ inside which
	\begin{equation}\label{eq:mt-ww-bounds-2}
		\D^{z,R_0/4}[r_i^{-1}X, Y_x] > \epsilon \text{\quad and \quad} \D^{z,R_0}[r_i^{-1}X_{x,r}, Y_x] < \frac{\epsilon}{8}
	\end{equation}
	Applying Lemma \ref{l:mt-ex-lb} as in the proof of \eqref{item:euc-apptan-equals-tan}, we see that
	\[ \limsup_{i\rightarrow\infty}s_i^{-1}\ex(B_X(x,s_i/2), X_{x,s_i}) > \frac{\epsilon}{8}. \qedhere\]
\end{proof}

\subsection{An Application to H{\"o}lder images} The problem of classifying sets that are H\"older rectifiable (that is, sets that are contained in the image of a H\"older curve) has attracted considerable attention since many well-known fractals can be parameterized in a H\"older way but not a Lipschitz one. See for example \cite{MM93,MM2000,BV,BNV19,BZ}. One of the reasons why the H\"older rectifiability problem is profoundly harder than its Lipschitz counterpart is that the tangents of H\"older curves can have much wilder behavior \cite{ShawV}. We end this section with the following theorem, which shows that for H\"older images, approximate tangents are exactly the tangents.

\begin{theorem}\label{thm:Holder}
	Let $s > 0,$ and $m,d\in\N$ with $0 < m < d$. If $f:\mathbb{B}^m\rightarrow \R^d$ is $\frac{s}{m}$-H{\"o}lder, then
	\begin{equation}
		\appTanAW(f(\mathbb{B}^m),x,s) = \TanAW(f(\mathbb{B}^m),x) \quad \text{for $\Haus^s$-a.e. $x\in f(\mathbb{B}^m)$.}
	\end{equation}
	If instead $f:\mathbb{B}^m\rightarrow X$ for some metric space $(X,d)$, then
	\begin{equation}
		\appTanGH(f(\mathbb{B}^m),x,s) = \TanGH(f(\mathbb{B}^m),x) \quad\text{for $\Haus^s$-a.e. $x\in f(\mathbb{B}^m)$.}
	\end{equation}
\end{theorem}

For the proof of Theorem \ref{thm:Holder} we need the next lemma.

\begin{lemma}\label{l:holder-packing-content}
	Let $X$ be a metric space, let $s>0$, and let $m\in\N$. Let $f:\mathbb{B}^m\rightarrow X$ be $(m/s)$-H{\"o}lder. The set $f(\mathbb{B}^m)$ has finite $s$-packing content.
\end{lemma}

\begin{proof}
	Let $B_X(x_i,r_i)$, $i\in\N$, be mutually disjoint balls with $x_1.x_2,\dots \in f(\mathbb{B}^m)$ and $r_i < \diam{f(\mathbb{B}^m)}$ for all $i\in\N$. There exists a constant $C>0$ depending only on $m$, $s$, and $f$, such that the balls $B^m(f^{-1}(x_i), Cr_i^{s/m})$ are mutually disjoint. By Remark \ref{rem:n-pack}, (recall that $\mathbb{B}^m$ has finite Lebesgue $m$-measure), $\mathbb{B}^m$ has finite $m$-packing content and there exists $H>0$ depending only on $m$ such that
	\[ \sum_{i\in\N} r_i^s = C^{-m}\sum_{i\in\N} (C r_i^{s/m})^m < C^{-m}H.  \]
	Therefore, $f(\mathbb{B}^m)$ has finite $s$-packing content.
\end{proof}

\begin{proof}[{Proof of Theorem \ref{thm:Holder}}]
	Lemma \ref{l:holder-packing-content} implies that $f(\mathbb{B}^m)$ has finite packing $s$-content in either case. Therefore, the results follow from applying Propositions \ref{p:euc-apptan-upgrade} and \ref{p:mt-apptan-upgrade}.
\end{proof}

\section{Rectifiable metric spaces admitting finite packing content}
In this section we give the proofs of Theorem \ref{thm:main-euc} and Theorem \ref{thm:main-met}. Using Theorem \ref{t:sob-low-reg}, it suffices to prove that $n$-rectifiable sets with finite $n$-packing content have a unique $n$-plane Attouch-Wets tangent (or a unique $n$-dimensional normed space Gromov-Hausdorff tangent) $\Haus^n$-almost everywhere. In both the Euclidean and metric cases, the strategy is roughly as follows.
\begin{enumerate}
	\item\label{item:approx-tan} At $\Haus^n$-a.e. $x\in E$, construct an approximate tangent $n$-plane (or approximate tangent normed $n$-space) that is also a strong approximate tangent.
	\item\label{item:inf-content}Use the upgrade criteria to upgrade these approximate tangents to unique true tangents $\Haus^n$-a.e.
\end{enumerate}
The approximate tangent from Step 1 is a well-known consequence of Rademacher's theorem in the Euclidean case and of Kirchheim's theorem in the metric case. The primary new contribution in this section is applying the results from previous sections to achieve Step 2.
\subsection{Approximate tangent planes}
We say that an $n$-plane $V\subset\R^d$ is an \emph{approximate tangent $n$-plane} to $E\subset \R^d$ at $x\in E$ simply if $V$ is an approximate $n$-tangent to $E$ at $x$.

We begin by showing that an $n$-rectifiable set $E\subset\R^d$ has a unique approximate tangent plane, which is also a strong approximate tangent, at $\mathcal{H}^n$-almost every point. Several definitions of approximate tangent planes exist in the literature, and the almost everywhere existence and uniqueness of these approximate tangent planes are well-known results (See \cite[Theorem 15.19]{Ma95} and \cite[Theorem 11.16]{Simon}). In practice, the differences between definitions are minor for $n$-rectifiable sets since the planes they produce are equal almost everywhere. Our definition of approximate tangent plane resembles a weak version of that given in \cite{Ba22}.
Before we get the existence of approximate tangent $n$-planes, we need the following topological lemma.
\begin{lemma}[\cite{BHS23} Lemma 2.3.9] \label{l:top}
	Let $0 < \delta < \frac{1}{2}$. Let $x\in \R^n$ and $r > 0$. Let $f:B^n(x,r)\rightarrow\R^n$ be a continuous function such that $|f(x) -x| \leq \delta r$ for all $x\in\partial B^n(x,r)$. Then $f(B^n(x,r)) \supset B^n(x,(1-\delta r))$.
\end{lemma}

\begin{lemma}\label{l:apptan-planes-exist}
	If $E\subset\R^d$ is $n$-rectifiable with $\Haus^n(E) < \infty$, then at $\cH^n$-a.e.-point $x\in E$, the set $E$ has a unique approximate tangent $n$-plane. This plane is a strong approximate $n$-tangent.
\end{lemma}

\begin{proof}
	Let $\{f_i\}$ be a countable collection of Lipschitz maps $f_i:\R^n\rightarrow\R^d$ such that $\Haus^n(E\setminus\bigcup_if_i(\R^n)) = 0$. Let $G\subset E$ be the set of points $x\in E$ such that
	\begin{enumerate}
		\item $\Theta^{*n}(E,x) \leq 1$,
		\item there exists $i > 0$ such that
		\begin{enumerate}
			\item $x$ is an $\Haus^n$-density point of $f_i(\R^n)$,
			\item there exists $y\in f_i^{-1}(x)$ such that $Df_i(y)$ exists and has full rank.
		\end{enumerate}
	\end{enumerate}
	By a Hausdorff measure density theorem \cite[Theorem 6.2]{Ma95}, Rademacher's theorem \cite[Theorem 7.3]{Ma95}, and a Sard-type theorem \cite[Theorem 7.6]{Ma95}, $G$ has full measure in $E$. Therefore, it suffices to fix $x\in G$ and show that there exists a unique approximate tangent of $E$ at $x$ that is an $n$-plane. For the rest of the proof, fix a Lipschitz function $f \equiv f_i$ and a point $y$ associated to $x$ as in the definition of $G$.
	
	Set $A_y:\R^n \to \R^d$ with $A_y(z) = f(y) + Df(y)(z-y) = x + Df(y)(z-y)$. Since $Df(y)$ has full rank, there exists $L_y > 1$ such that $A_y$ is an $L_y$-bi-Lipschitz map. Let $\epsilon = (10L_y)^{-1}$ and let $r_0 > 0$ be such that
	\begin{equation*}
		|f(z) - (f(y) + Df(y)(z-y))|\leq \epsilon |z-y|
	\end{equation*}
	for all $z\in B^n(y,r_0)$.  Let $V_x = A_y(\R^n)$ and define $\pi_{V_x}:\R^d\rightarrow V_x$ to be the orthogonal projection onto $V_x$. Consider the continuous map
	\begin{equation*}
		\td{f} := A_y^{-1}\circ\pi_{V_x}\circ f:\R^n\rightarrow\R^n.
	\end{equation*}
	If $z\in\partial B^n(y,r_0)$, then
	\begin{align*}
		|\td{f}(z) - z| &= |A_y^{-1}(\pi_{V_x}(f(z))) - A_y^{-1}(A_y(z))|\\
&\leq L_y|\pi_{V_x}(f(z)) - A_y(z)|\\
&= L_y|\pi_{V_x}(f(z) - A_y(z))|\\
&\leq L_y|f(z)-A_y(z)|\leq L_y\epsilon|z-y|\leq  \frac{r_0}{10}.
	\end{align*}
	Thus, by Lemma \ref{l:top}, $\td{f}(B^n(y,r_0))\supset B^n(y,\frac{1}{2}r_0)$. This gives
	\begin{align}\label{e:pi-surj}
		\pi_{V_x}(f(B^n(y,r_0))) &\supset A_y(B^n(y,r_0/2)) \supset B^d(x,L_y^{-1}r_0/2)\cap V_x.
	\end{align}
	We define $E_{x,r} = f(B^n(y,r_0))$ for all $r > 0$. Using the fact that $V_x$ is a cone centered at $x$ and then the definition of $Df(y)$, it is straightforward to show that for any $R > 0$,
	\begin{align}\label{e:app-WW}
		\lim_{r\rightarrow0}\D^{\bz,R}[r^{-1}(E_{x,r}-x),V_x-x] &=\nonumber \lim_{r\rightarrow0}\D^{\bz,Rr}[E_{x,r}-x, r(V_x-x)]\\ \nonumber
		&= \lim_{r\rightarrow0}\D^{\bz,Rr}[E_{x,r}-x, V_x-x]\\
		&= \lim_{r\rightarrow 0}\D^{x,Rr}[E_{  x,r}, V_x] = 0.
	\end{align}
	This verifies the strong version of \eqref{e:euc-apptan-distance} holds. For the density conclusion, recall that we chose $x$ to be a density point for $f(\R^n)$ and \eqref{e:pi-surj} combined with the vanishing relative Walkup-Wets distance implies that, for any $R > 0$,
	\begin{align*}
		\Haus^n(E_{x,r}\cap B^d(x,Rr)) &\geq \Haus^n(\pi_{V_x}(f(B^n(y,r_0)))\cap B^d(x,Rr))\\
		&\geq \Haus^n(V_x\cap B^d(x,(1-\delta(Rr))Rr))
	\end{align*}
	where we take $r < (10L_y)^{-1}r_0$ and $\delta(r)\rightarrow 0 $ as $r\rightarrow 0$. This means that
	\begin{equation}\label{e:bx-haus}
		\lim_{r\rightarrow0}\frac{\cH^n(E_{x,r}\cap B^d(x,Rr))}{(2Rr)^n} = \lim_{r\rightarrow0}\frac{\cH^n(E\cap B^d(x,Rr))}{(2Rr)^n} = 1,
	\end{equation}
	which shows that $\lim_{r\rightarrow0}(2Rr)^{-n}\cH^n(E\cap B^d(x,Rr)\setminus E_{x,r}) = 0$.	This shows that $V_x-x$ is an approximate tangent $n$-plane that is also a strong approximate tangent.
	
	We now show that $V_x$ is unique. Towards a proof by contradiction, suppose that there exists an affine $n$-plane $V_x'\not=V_x$ such that $V_x'-x$ is an approximate tangent $n$-plane to $E$ at $x$. Let $\epsilon > 0$ be such that $\sD^{x,r}[V_x,V_x'] > 10\epsilon$ for all $r > 0$ and let $r_k\rightarrow0$ and subsets $E_{x,r_k}'$ for $k > 0$ satisfy conditions \eqref{e:euc-apptan-density} and \eqref{e:euc-apptan-distance} for $V_x'$. Using \eqref{e:pi-surj} and \eqref{e:app-WW}, we can conclude that for any sufficiently small $r > 0$ there exists $v\in V_x$ with $|v-x| = r/2$, $\dist(v,V_x') \geq 5\epsilon r$, and
	\begin{equation}
		\pi_{V_x}(E_{x,r}\cap B^d(V_x,n^{-1}\epsilon r)) \supset V_x\cap B^d(v,\epsilon r).
	\end{equation}
	In particular, this shows that
	\begin{equation}
		\Haus^n(E_{x,r}\cap B^d(v,\epsilon r)) \gtrsim_n \epsilon^nr^n \hbox{\quad for all sufficiently small $r>0$.}
	\end{equation}
	Since $E_{x,r}\subset E$, this means we can argue similarly as in \eqref{e:app-WW} to find $r_k > 0$ small enough such that
	\[\sD^{x,r_k}[E_{x,r_k}',V_x'] = \sD^{\bz,1}[r_k^{-1}(E'_{x,r_k}-x),V_x'-x] < \epsilon,\]
	and there exists $y\in E_{x,r_k}'\cap (E_{x,r_k}\cap B^d(v,\epsilon r_k))$. Since $y\in B^d(x,r_k)$, this gives the following contradiction:
	\begin{equation*}
		\epsilon r_k > \ex(E_{x,r_k}',V_x') \geq \dist(y,V_{x}') \geq \dist(v,V_x') - |y-v| \geq 4\epsilon r_k. \qedhere
	\end{equation*}
\end{proof}
\subsection{Approximate tangent normed spaces}
We say that a normed space $(\R^n,\|\cdot\|)$ is an \emph{approximate tangent normed $n$-space} to a metric space $X$ at $x\in X$ simply if it is an approximate $n$-tangent to $X$ at $x$. For a related notion of tangent space and a related theorem characterizing rectifiable metric spaces in terms of that notion, see \cite[Theorem 1.2]{Ba22}. The existence of these tangents in $n$-rectifiable metric spaces is largely handled by the following theorem due to Kirchheim.
\begin{theorem}[{\cite[Theorem 9]{Ki94}}]\label{t:kirch}
	Let $X$ be $n$-rectifiable with $\Haus^n(X)<\infty$. For $\Haus^n$-a.e. $x\in X$, there exists a norm $\|\cdot\|_x$ on $\R^n$, a map $\phi_x:X\rightarrow\R^n$, and a closed set $X_x\subset X$ such that $\phi_x(x) = \bz,$
	\begin{equation}\label{e:ax-dens}
		\lim_{r\rightarrow 0}\frac{\Haus^n(X_x\cap B_X(x,r))}{(2r)^n} = \lim_{r\rightarrow 0}\frac{\Haus^n(B_X(x,r))}{(2r)^n} = 1,
	\end{equation}
	and
	\begin{equation}\label{e:bx-isom}
		\limsup_{r\rightarrow 0}\left\{\left|1 - \frac{\|\phi_x(y) - \phi_x(z)\|_x}{d_X(y,z)}\right| : y,z\in X_x\cap B_X(x,r),\ y\not= z\right\} = 0.
	\end{equation}
	In particular, for every $\delta > 0$, there exists some $r_\delta > 0$ such that for every $r < r_\delta$,
	there exists a $\delta r$-pGHA from $X_x\cap B(x,r)$ to $B_{\|\cdot\|_x}(\bz,r)$.
\end{theorem}

Although Kirchheim's result gives the existence and uniqueness of approximate tangents in the sense of Theorem \ref{t:kirch}, the proof of uniqueness for our weaker notion of approximate tangent requires more work than what is already present in \cite{Ki94}.

\begin{lemma}\label{c:apptan-normed-spaces-exist}
	If $X$ is an $n$-rectifiable metric space with $\cH^n(X) < \infty$, then at $\cH^n$-a.e. $x\in X$ the space $X$ has an approximate tangent normed $n$-space that is unique up to isometries. This normed space is a strong approximate $n$-tangent.
\end{lemma}
\begin{proof}
	It suffices to get the existence of approximate tangents at every point of the good set $G$ produced by Theorem \ref{t:kirch}. Fix $x\in G$ and let $X_x,\phi_x,$ and $\|\cdot\|_x$ be as in Theorem \ref{t:kirch}. We claim that $(\R^n,\|\cdot\|_x,\bz)$ is our desired tangent at $x$ using $X_{x,r}= X_x$ as ``substantial sets'' for any $r > 0$. Indeed, equation \eqref{e:ax-dens} implies the strong version of \eqref{e:mt-apptan-density} while the final statement of the theorem implies
	\[ \lim_{r\rightarrow0}\xi_{X_{x},(\R^n,\|\cdot\|_x)}(x,\bz,r) = 0.\]
	Since $(\R^n,\|\cdot\|_x,\bz)$ is a metric cone, the final statement of Lemma \ref{l:mt-cones-and-pGHAs} implies the strong version of \eqref{e:mt-apptan-distance} holds, so $(\R^n,\|\cdot\|_x,\bz)$ is a strong approximate $n$-tangent.
	
	We now show that $(\R^n,\|\cdot\|_x,\bz)$ is the unique approximate tangent $n$-normed space at $x$ up to isometry. Let $\|\cdot\|_x'$ be a norm such that $(\R^n,\|\cdot\|_x',\bz)$ is an approximate tangent $n$-normed space at $x$, and let $r_i\rightarrow0$ and $X_{x,r_i}\subset X$ be subsets satisfying \eqref{e:mt-apptan-density} and \eqref{e:mt-apptan-distance}. We will construct an isometry $f:B_{\|\cdot\|_x}(\bz,1)\rightarrow B_{\|\cdot\|_x'}(\bz,1)$.
	
	Using \eqref{e:ax-dens} and \eqref{e:mt-apptan-density} with for $X_{x,r_i}'$, we see that
	\begin{equation}
		\lim_{r\rightarrow0} \frac{\Haus^n(B_X(x,Rr_i)\setminus(X_{x,r_i}\cap X_{x,r_i}'))}{(2Rr_i)^n} = 0 \hbox{\quad for every $R > 0$}.
	\end{equation}
	Using the bi-Lipschitz estimate \eqref{e:bx-isom}, this gives the following statement: For every $R, \epsilon > 0$, there exists $r_\epsilon > 0$ for which $r_i < r_\epsilon$ implies the existence of a $\epsilon^2 Rr_i$-pGHA $\varphi_{\epsilon,Rr_i}:B_{\|\cdot\|_x}(\bz,Rr_i)\rightarrow B_X(x,Rr_i)\cap(X_{x,r_i}\cap X_{x',r_i})$. As long as $r_i$ is sufficiently small, Lemma \ref{l:mt-cones-and-pGHAs} \eqref{item:DGH-implies-xi} and \eqref{e:mt-apptan-distance} give the existence of a $3\epsilon^2 (r_i/\epsilon)$-pGHA $\psi_{\epsilon,r_i/\epsilon}:B_X(x,r_i/\epsilon)\cap X_{x,r_i}' \rightarrow B_{\|\cdot\|_x'}(\bz,r_i/\epsilon)$. For sufficiently small $r_i$, we set $R = 1/\epsilon, s_i = r_i/\epsilon$ and combine these mappings to get a mapping $f_\epsilon:B_{\|\cdot\|_x}(\bz,1)\rightarrow B_{\|\cdot\|_x'}(\bz,1)$ given by
	\begin{equation}
		f_\epsilon(u) = \frac{1}{s_i}\psi_{\epsilon,s_i}(\varphi_{\epsilon,s_i}(s_i u)).
	\end{equation}
	Observe that $\varphi_{\epsilon,s_i}(\bz) = x$ and $\psi_{\epsilon,s_i}(x) = \bz$ so that $f_\epsilon(\bz) = \bz$. Since $\psi_{\epsilon,s_i}$ and $\varphi_{\epsilon,s_i}$ are both $3\epsilon^2 s_i$-pGHAs, a straightforward argument using the triangle inequality shows that
	\begin{equation}
		\bigg|\|f_\epsilon(u)-f_\epsilon(v)\|_x' - \|u-v\|_x\bigg| \leq 6\epsilon^2 \hbox{\quad for all $u,v\in B_{\|\cdot\|_x}(\bz,1)$}.
	\end{equation}
	We will use these maps for decreasing $\epsilon$ to produce an isometry $f$ between $B_{\|\cdot\|_x}(\bz,1)$ and $B_{\|\cdot\|_x'}(\bz,1)$.
	
	Let $\epsilon_i \rightarrow 0$ and consider the sequence of maps $\{f_{\epsilon_i}\}$. Choose a countable dense set $N\subset B_{\|\cdot\|_x}(\bz,1)$. Via a standard diagonal argument, we can find a subsequence $i_k$ such that
	\begin{equation}
		\lim_{k}f_{\epsilon_{i_k}}(w) \hbox{\quad exists for all $w\in N$}.
	\end{equation}
	We define a map $f:B_{\|\cdot\|_x}(\bz,1)\rightarrow B_{\|\cdot\|_x'}(\bz,1)$ by
	\begin{equation}
		f(u) = \begin{cases}
			\lim_{k}f_{\epsilon_{i_k}}(u) &\hbox{ if $u\in N$},\\
			\lim_{n}f(u_n) &\hbox{ for any sequence $u_n\rightarrow u$ with $u_n\in N$ when $u\not\in N$}.
		\end{cases}
	\end{equation}
	Using the triangle inequality, it is standard to show that $f$ is well-defined and
	\begin{equation}\label{e:isom-embedding}
		\|f(u)-f(v)\|_x' = \|u-v\|_x \hbox{\quad for all $u,v\in B_{\|\cdot\|_x}(\bz,1)$}.
	\end{equation}
	
	We now show that $f$ is surjective on $\partial B_{\|\cdot\|_x}(\bz,r)$ for any $0 < r \leq 1$. First, note that $f(\bz) = \bz$ and \eqref{e:isom-embedding} imply $f$ maps $\partial B_{\|\cdot\|_x}(\bz,r)$ into $\partial B_{\|\cdot\|_x'}(\bz,r)$. Towards a proof by contradiction, suppose that $f:\partial B_{\|\cdot\|_x}(\bz,r)\rightarrow \partial B_{\|\cdot\|_x'}(\bz,r)$ is not surjective. This gives a homeomorphic embedding $g:\mathbb{S}^{n-1}\rightarrow\mathbb{S}^{n-1}\setminus \{p\}$ for some $p\in\mathbb{S}^{n-1}$. This leads to the contradiction $g(\mathbb{S}^{n-1}) = \mathbb{S}^{n-1}\setminus\{p\}$ since $g(\mathbb{S}^{n-1})$ is open and closed in the connected space $\mathbb{S}^{n-1}\setminus\{p\}$. This proves that $f$ is surjective, which shows it is an isometry between $B_{\|\cdot\|_x}(\bz,1)$ and $B_{\|\cdot\|_x'}(\bz,1)$.
\end{proof}

\subsection{Euclidean space-valued Sobolev mappings}
In this section, we prove our results on the almost everywhere existence of tangent $n$-planes.
\begin{theorem}\label{t:euc-flat-tangents}
	If $E\subset \R^d$ is $n$-rectifiable and has finite $n$-packing content, then for $\Haus^n$-a.e. $x\in E$, there exists an $n$-plane $V_x$ such that
	$$\Tan_\AW(E,x) = \{V_x\}.$$
\end{theorem}

\begin{proof}
	Since $E$ is $n$-rectifiable, Lemma \ref{l:apptan-planes-exist} implies that $E$ has an approximate tangent $n$-plane $V_x$ that is a strong approximate tangent at $\Haus^n$-a.e. $x\in E$. Applying Proposition \ref{p:euc-apptan-upgrade}, we get the conclusion.
\end{proof}

We now conclude the subsection with some straightforward corollaries.

\begin{corollary}\label{cor:euc-sobolev-image-tangents}
	Let $n < p \leq \infty$ and  $k\in\N$. Suppose that $f_i\in N^{1,p}(\mathbb{B}^n,\R^d)$ are continuous for $i\in \{1,\dots, k\}$. If $A = \bigcup_{i=1}^kf_i(\mathbb{B}^n)$, then for $\mathcal{H}^n$-a.e. $x\in A$, there exists an $n$-plane $V_x$ such that
	$$\Tan_\AW(A,x) = \{V_x\}$$
\end{corollary}

\begin{proof}
	By Theorem \ref{t:sob-low-reg}, $A$ is $n$-rectifiable and has finite $n$-packing content, so $A$  satisfies the hypotheses of Theorem \ref{t:euc-flat-tangents}.
\end{proof}

\begin{corollary}\label{cor:euc-rect-set-tangents}
	Let $E\subset \R^d$ be $n$-rectifiable with finite $\mathcal{H}^n$ measure. For any $\epsilon > 0$, there exists $\tilde{E}\subset E$ such that $\Haus^n(E\setminus\tilde{E}) < \epsilon$, and for $\mathcal{H}^n$-a.e. $x\in \tilde{E}$, there exists an $n$-plane $V_x$ such that
	$$\Tan_\AW(A,x) = \{V_x\}$$
\end{corollary}

\begin{proof}
	Let $\{f_i\}_{i=1}^N$ be a finite collection of Lipschitz maps $f_i:\mathbb{B}^n\rightarrow\R^d$ such that $\Haus^n(E\setminus\bigcup_{i=1}^N f_i(\mathbb{B}^n)) < \epsilon)$. Set $\tilde{E} = E\cap\bigcup_{i=1}^N f_i(\mathbb{B}^n)$. Since $\tilde{E}$ is a subset of a finite union of Lipschitz images, $\tilde{E}$ has finite $n$-packing content. Therefore, we can apply Theorem \ref{t:euc-flat-tangents}.
\end{proof}

\subsection{Metric space-valued Sobolev mappings}
In this subsection, we prove our results on the existence of tangent $n$-dimensional normed spaces.
\begin{theorem}\label{t:flat-tangents}
	If $X$ is a metric space that is $n$-rectifiable and has finite $n$-packing content, then for $\mathcal{H}^n$-a.e. $x\in X$, there exists a norm $\|\cdot\|_x$ on $\R^n$ such that
	\[ \Tan_\GH(X,x) = \{[(\R^n,\|\cdot\|_x,\bz)]\}.\]	
\end{theorem}

\begin{proof}
	Since $X$ is $n$-rectifiable, Lemma \ref{c:apptan-normed-spaces-exist} implies that $X$ has an approximate tangent normed $n$-space $(\R^n,\|\cdot\|_x,\bz)$ that is a strong approximate tangent at $\Haus^n$-a.e. $x\in X$. Applying Proposition \ref{p:mt-apptan-upgrade}, we get the conclusion.
\end{proof}
The following corollaries are metric versions of Corollaries \ref{cor:euc-sobolev-image-tangents} and \ref{cor:euc-rect-set-tangents}.
\begin{corollary}\label{cor:met-sobolev-lip-image-tangents}
	Let $n < p \leq \infty$ and $k\in\N$. Suppose that $f_i\in N^{1,p}(\mathbb{B}^n,X)$ are continuous for $i\in\{1,\dots,k\}$. If $A = \bigcup_{i=1}^kf_i(\mathbb{B}^n)$, then for $\mathcal{H}^n$-a.e. $x\in X$, there exists a norm $\|\cdot\|_x$ on $\R^n$ such that
	\[ \Tan_\GH(A,x) = \{[(\R^n,\|\cdot\|_x,\bz)]\}.\]	
\end{corollary}

\begin{proof}
	By Theorem \ref{t:sob-low-reg}, $A$ is $n$-rectifiable and has finite $n$-packing content, so $A$ satisfies the hypotheses of Theorem \ref{t:flat-tangents}.
\end{proof}

\begin{corollary}\label{cor:rect-set-tangents}
	Let $X$ be an $n$-rectifiable metric space with finite $\mathcal{H}^n$ measure. For any $\epsilon > 0$, there exists $\tilde{X}\subset X$ such that $\Haus^n(X\setminus\tilde{X}) < \epsilon$ and for $\mathcal{H}^n$-a.e. $x\in \tilde{X}$, there exists a norm $\|\cdot\|_x$ on $\R^n$ such that
	\[ \Tan_\GH(\tilde{X},x) = \{[(\R^n,\|\cdot\|_x,\bz)]\}.\]	
\end{corollary}

\begin{proof}
	Without loss of generality, we can assume that $X\subset V$, where $V$ is a Banach space, and that $\{f_i\}_{i=1}^N$, is a collection of Lipschitz maps $f_i:\mathbb{B}^n\rightarrow\ell_\infty$ such that $\Haus^n(X\setminus\bigcup_{i=1}^N f_i(\mathbb{B}^n)) < \epsilon$. Set $\tilde{X} = X\cap \bigcup_{i=1}^Nf_i(\mathbb{B}^n)$. Since $\tilde{X}$ is a subset of a finite union of Lipschitz images, it has finite $n$-packing content. Therefore, we can apply Theorem \ref{t:flat-tangents}.
\end{proof}

\section{Related Examples}

In this section, we illustrate the sharpness of our results. The first example is a closed $n$-rectifiable set which has approximate tangent planes almost everywhere, but no tangent $n$-planes.
\begin{example}\label{ex:approx-vs-true}
	Fix $n\in\N$ and for each $k\in\N$ define $\tilde{V}_k = 2^{1-k}\{0,\dots, 2^{k-1}\}^n \subset [0,1]^n$. Define also $V_1 = \tilde{V}_1$ and $V_k = \tilde{V}_k \setminus \tilde{V}_{k-1}$ for $k\geq 2$. Let
	\[ X := \left([0,1]^n\times \{0\}\right) \cup \bigcup_{k=1}^{\infty} \bigcup_{x\in V_k} \{x\}\times[-1/k,1/k] \subset \R^{n+1}.\]
	It is not hard to see that that $X$ is closed, $n$-rectifiable, and $\mathcal{H}^n(X) = \mathcal{H}^n([0,1]^n\times \{0\}) = 1$. However, for all $p \in (0,1)^n\times \{0\}$ (hence for $\mathcal{H}^n$-a.e. $p\in X$) we have that $\TanAW(X,p) = \{\R^{n+1}\}$.
\end{example}

The second example shows that in the conclusions of our results, tangents cannot be replaced by \emph{pseudo-tangents}. Recall that $T$ is a pseudo-tangent of $X\subset \R^n$ at $p\in X$ if there exists a sequence of points $(p_m)_{m\in\N}$ in $X$ and a sequence $(r_m)_{m\in\N}$ of positive numbers converging to zero such that $r_m^{-1}(X-x_m)$ converges to $T$ in the Attouch-Wets topology \cite[Definition 3.1]{BL15}.

\begin{example}\label{ex:pseudot}
	Let $n\in\N$, and let $\{\mathcal{C}_{i_1\cdots i_k} : k\in\N, i_1,\dots,i_k \in \{1,\dots,2^n\}\}$ be a collection of closed $n$-cubes in $[0,1]^n$ such that
	\begin{enumerate}
		\item if $w=i_1\cdots i_k \in \{1,\dots,2^n\}^k$ and if $i,j \in \{1,\dots,2^n\}$ are distinct, then
		\begin{align*}
			\mathcal{C}_{w i} \subset \mathcal{C}_{w}, \quad
			\mathcal{C}_{w i} \cap \partial\mathcal{C}_{w} = \emptyset, \quad\text{and}\quad
			\mathcal{C}_{w i} \cap \mathcal{C}_{w j} = \emptyset,
		\end{align*}
		\item
		\[ \lim_{k\to \infty} \sup_{w \in \{1,\dots,2^n\}^k} \diam{\mathcal{C}_{w}} =0,\]
		\item the Cantor set
		\[ E := \bigcap_{k=1}^{\infty} \bigcup_{w \in \{1,\dots,2^n\}^k} \mathcal{C}_{w} \]
		has positive Lebesgue $n$-measure.
	\end{enumerate}
	
	For each $k\in\N$ and $w= i_1\cdots i_k \in \{1,\dots,2^n\}^k$ fix $p_{w} \in [0,1]^n$ and $r_{w} \in (0,2^{-k})$ such that
	\[ B^n(p_{w}, 2r_{w}) \subset \mathcal{C}_{w} \setminus \bigcup_{i\in\{1,\dots,2^n\}} \mathcal{C}_{wi}. \]
	Note that if $w,w'$ are distinct, then $B^n(p_w,r_w)\cap B^n(p_{w'},r_{w'}) = \emptyset$.
	
	Define now the function $f:[0,1]^n \to \R$ by
	\[ f(x) = \sum_{k\in\N} \sum_{w \in \{1\,\dots,2^n\}^k} \max\left\{0,r_w - |x-p_w|\right\}. \]
	It is not hard to see that $f$ is a Lipschitz function, which means that the graph $G = \{(x,f(x)) : x\in [0,1]^n\}$ of $f$ in $\R^{n+1}$ admits a bi-Lipschitz parameterization
	\[ F : [0,1]^n \to G, \qquad F(x) = (x,f(x)).\]
	Therefore, $F(E)$ has positive $\Haus^n$-measure. We claim that for all $x \in F(E)$, there exists a pseudo-tangent $T$ of $G$ at $x$ which is not an $n$-plane.
	
	Towards this end, fix $x\in E$ and let $(w_k)_{k\in\N}$ be a sequence of words such that $w_k \in \{1,\dots,2^n\}^k$ for each $k\in\N$ and
	\[ \{x\} = \bigcap_{k=1}^{\infty} \mathcal{C}_{w_k}.\]
	Let also $x_k = (p_{w_k},r_k) \in G$. By property (2) in the definition of sets $\mathcal{C}_w$,
	\[ |x_k - (x,0)| \leq r_{w_k} + \diam{\mathcal{C}_{w_k}} \xrightarrow{k\to \infty} 0.\]
	There exist positive integers $k_1<k_2<\cdots$ such that $r_{k_j}^{-1}(G-x_j) \to T$ for some closed set $T \in \mathfrak{C}(\R^{n+1})$ in the Attouch-Wets topology. To complete the proof of the claim, note that for all $j\in\N$, the cone
	\[ \{(x,1-|x|) : x \in B^n(\textbf{0},1)\} \subset B^n(\textbf{0},1)\times[0,1]\]
	is contained in $r_{k_j}^{-1}(G-x_j)$, so it is also contained in $T$. Therefore, $T$ is not isometric to an $n$-plane.
\end{example}

We now give an example showing that in Corollary \ref{cor:euc-rect-set-tangents}, one cannot take $\tilde{E} \subset E$ such that $\mathcal{H}^n(E\setminus \tilde{E}) =0$.

\begin{example}\label{ex:rect-no-tan-planes}
	Let $1 \leq n < d$, $D = [0,1]^n\times\{0\}^{d-n}\subset\R^d$, and let let $\cW$ be the standard dyadic Whitney decomposition of $[0,1]^d$ with respect to $D$. Let $X = \{x_i\}_i$ be an enumeration of the centers of the cubes in $\cW$. Let $r_i>0$ be a sequence of radii such that $\sum_i r_i^n < \infty$ and let $B^n(x_i,r_i)$ denote the closed $n$-disk in the first $n$-coordinates of radius $r_i$ centered at $x_i$. Define $E = D \cup \bigcup_{i=1}^\infty B^n(x_i,r_i)$. Then $\Haus^n(E) < \infty$, $E$ is compact, $n$-rectifiable, and yet 
	if
	$\tilde{E}\subset E$ with $\Haus^n(E\setminus \tilde{E}) = 0$, then $\tilde{E}$ has no $n$-plane tangents at any point of $D\cap \tilde{E}$.
\end{example}

Finally, we give an example showing that the value of $p$ in Corollary \ref{cor:euc-sobolev-image-tangents} is sharp when $n \geq 2$ by constructing a function $F\in W^{1,n}([0,1]^n,\R^{n+1})$ satisfying Lusin's condition N whose image has no $n$-plane tangents on a fat cantor set $K$ in $[0,1]^n\times\{0\}\subset\R^{n+1}$.

\begin{example}\label{ex:psharp}
	Let $(\lambda_i)_{i\in\N}$ be a sequence in $(0,1)$ such that $\prod_{i=1}^\infty(1-\lambda_i) > 0$. We let $K_0$ be the union of the $2^{n}$ corner cubes inside $[0,1]^n$ of side length $\frac{1-\lambda_0}{2}$. Given $K_j$ for any $j\geq 0$, we define $K_{j+1}$ to be the union of the $2^n$ corner cubes of side length $\frac{1-\lambda_{j+1}}{2}\ell(Q)$ where $Q$ is any $j$-th order corner cube in the definition of $K_j$. Let $K = \bigcap_{j=1}^\infty K_j$ and observe that $\scL^n(K) > 0$.
	
	We now construct a function $f\in W^{1,n}([0,1]^n,\R)$ and define our desired function $F$ as its graph by setting $F(x) = (x,f(x))$. The idea to construct $F$ is to carefully ``poke out'' small balls in the complement of $K$ into the $(n+1)$-th coordinate direction to height $\ell(Q)$. For each corner cube $Q$ in $K_j$ in the above construction, let $B_Q = B^n(x_Q,\frac{\lambda_{j+1}}{2}\ell(Q))$ where $x_Q$ is the center of $Q$ so that $B_Q\subset[0,1]^n\setminus K_{j+1}$ and $B_Q\cap B_{Q'} =\varnothing$ for $Q\not=Q'$. Let $\scQ$ be the collection of all corner cubes in the construction and let $\{\alpha_Q\}_{Q\in\scQ}$ be a sequence of numbers in $(0,\frac12)$ so that, if $Q$ is a $k$-th generation cube, then
	\begin{equation}\label{eq:aq}
		|\log(\alpha_Q)| \geq \left( 2^{nk} k^2 \ell(Q)^n\right)^{\frac1{n-1}}.
	\end{equation}
	Define $f:[0,1]^n \to \R$ by
	\begin{equation*}
		f(x) = \begin{cases}
			0 & x\in [0,1]^n \setminus \bigcup_{Q\in\scQ}B_Q,\\
			\ell(Q)\displaystyle\frac{\log\left(|x-x_Q|r(B_Q)^{-1}\right)}{\log\left(\alpha_Q\right)} & x\in B_{Q}\setminus\alpha_QB_Q,\ Q\in\scQ,\\
			\ell(Q) & \ x\in\alpha_QB_Q,\ Q\in\scQ.
		\end{cases}
	\end{equation*}
	We note that $f$ is continuous, $\scL^n$-a.e. differentiable, and for $\scL^n$-a.e. $x\in[0,1]^n$,
	\begin{equation*}
		|\nabla f(x)| = \begin{cases}
			0 & x\in [0,1]^n \setminus \left(\bigcup_{Q\in\scQ}B_Q\setminus\alpha_QB_Q\right),\\
			\displaystyle\frac{\ell(Q)}{|\log\left(\alpha_Q\right)||x-x_Q|} & x\in B_{Q}\setminus\alpha_QB_Q,\ Q\in\scQ.
		\end{cases}
	\end{equation*}
	This gives
	\begin{align*}
		\int_{[0,1]^n}|\nabla f|^n = \sum_{Q\in\scQ}\int_{B_Q\setminus\alpha_QB_Q}|\nabla f|^n &\simeq_n \sum_{Q\in\scQ}\frac{\ell(Q)^n}{|\log(\alpha_Q)|^n}\int_{\alpha_Qr(B_Q)}^{r(B_Q)}\frac{1}{r}dr\\
		&=\sum_{Q\in\scQ}\frac{\ell(Q)^n}{|\log(\alpha_Q)|^{n-1}}< \infty,
	\end{align*}
	where the final inequality follows from \eqref{eq:aq}. It follows that the function $F(x) = (x,f(x))$ satisfies $F\in W^{1,n}([0,1]^n,\R^{n+1})$ and that $F$ is a homeomorphism from $[0,1]^n$ onto its image. We know that the image of $F$ does not have $n$-plane tangents at any point of $K$, since for every $x\in K$ and $0 < r < 1$, $B^{n+1}(x,r)\cap F([0,1]^n)$ contains a large subset of $[0,1]^n\times\{0\}$ as well as a spike of height $\ell(Q)$ for some $Q$ with $\ell(Q)\simeq r$.
\end{example}

\bibliographystyle{alpha}
\bibliography{bib-file-RST}

\end{document}